\documentclass[12pt]{article}
\usepackage{amssymb,amsthm,hyperref}
\newtheorem{thm}{Theorem}[section]

 \newtheorem{lem}{Lemma}[section]
 \newtheorem{prop}{Proposition}[section]
 \newtheorem{defn}{Definition}[section]%是否是全文计数？
\theoremstyle{remark}
\newtheorem{rem}{Remark}[section]

\title{Local well-posedness for Euler-Poisson fluids with non-zero heat conduction}
\author{Jiang Xu\thanks{E-mail: jiangxu\underline{ }79@yahoo.com.cn,\ jiangxu\underline{ }79@nuaa.edu.cn}\\
\small{\textit{Department of Mathematics}},
\\ \small{\textit{Nanjing
University of Aeronautics and Astronautics}}, \\
\small{\textit{Nanjing 211106, P.R.China}}}
\date{}
\begin{document}
\maketitle{}
\begin{abstract}
We consider the multidimensional Euler-Poisson equations with
non-zero heat conduction, which consist of a coupled
hyperbolic-parabolic-elliptic system of balance laws. We make a deep
analysis on the coupling effects and establish a local
well-posedness of classical solutions to the Cauchy problem
pertaining to data in the critical Besov space. Proof mainly relies
on a standard iteration argument. To achieve it, a new Moser-type
inequality is developed by the Bony' decomposition.
\end{abstract}

\hspace{-0.5cm}\textbf{Keywords:} \small{local well-posedness, Euler-Poisson equations, classical solutions, Besov spaces}\\

\hspace{-0.5cm}\textbf{AMS subject classification:} \small{35M10;
35Q35; 76X05}
\section{Introduction and main results}
The ongoing miniaturization of semiconductor devices, some high
field phenomena such as hot electron effects, impact ionization and
heat generation appear inside the devices. The traditional
drift-diffusion model employed for numerical simulation does not
provide an adequate description of these effects. Consequently, the
hydrodynamical model for semiconductors was introduced, which can be
derived from the Boltzmann equation by moment method based on the
shifted Maxwellian ansatz for the equilibrium phase space
distribution. Precisely, the hydrodynamical model takes the form of
the following compressible Euler-Poisson equations (see, e.g.,
\cite{MRS}):
\begin{equation}
\left\{
\begin{array}{l}\partial_{t}n+\mathrm{div}(n\textbf{u})=0,\\[1mm]
\partial_{t}(n\textbf{u})+\mathrm{div}(n\textbf{u}\otimes\textbf{u})+\nabla
P
 =n\nabla\mathit\Phi-\frac{n\textbf{u}}{\tau_{p}},
 \\[1mm]
\partial_{t}W+\mathrm{div}(W\textbf{u}+P\textbf{u})-\mathrm{div}(\kappa\nabla \mathcal{T})
=n\textbf{u}\cdot\nabla\mathit\Phi
-\frac{W-\overline{W}}{\tau_{w}},\\[1mm]
\lambda^2\Delta\mathit\Phi=n-\bar{n},
 \end{array} \right.\label{R-E1}
\end{equation}
for $(t,x)\in [0,+\infty)\times \mathbb{R}^{N}(N=2,3)$. Here,
$n(t,x)>0$ denotes the electron density, $\textbf{u}(t,x)\in
\mathbb{R}^{N}$ electron velocity and $W(t,x)$ energy density.
$\mathit\Phi=\mathit\Phi(t,x)$ represents the electrostatic
potential generated by the Coulomb force from electrons and
background ions. $P=n\mathcal{T}$ is the pressure of electron fluid
where $\mathcal{T}(t,x)$ is the temperature of electrons. The energy
density $W$ satisfies
$W=\frac{n|\textbf{u}|^2}{2}+\frac{P}{\gamma-1}(\gamma>1)$ and
$\overline{W}=\frac{n\mathcal{T}_{L}}{\gamma-1}$ is the ambient
device energy, where $\mathcal{T}_{L}>0$ is a given ambient device
temperature. $\bar{n}>0$ is the doping profile which stands for the
density of fixed, positively charged background ions. The scaled
coefficient $\tau_{p},\tau_{w}$ and $\lambda$ are the momentum
relaxation-time, energy relaxation-time and the Debye length,
respectively. The coefficient $\kappa$ is the heat conductivity,
which generally depends on the electron density and temperature. For
the sake of simplicity, we assume it to be one constant.

The full hydrodynamical model (\ref{R-E1}) for the balance laws of
the density, velocity, temperature and the electric potential
consists of a quasi-linear hyperbolic-parabolic-elliptic system,
which contains the damping relaxation, heat conduction and electric
dissipation. The interaction of these special effects makes it
complicated to understand the qualitative behavior of solutions,
however, many efforts were made by various authors, see
\cite{A,ABN,ACJP,CJZ,DM,G,GS,GN,HJZ,HMW,HW,L,LNX,MN,W,WC,X1,XY2,Y}
and the references therein, for issues of well-posedness of
steady-state solutions or classical, large time behavior and
singular limit problems.

In this paper, we are concerned with the well-posedness of classical
solutions starting with smooth initial data under the coupled
effects. In the one space dimension, Chen, Jerome, and Zhang
\cite{CJZ} first considered the initial boundary problem of
(\ref{R-E1}) and established the local existence of smooth
solutions. Furthermore, they showed that smoothness in local
solutions can be extended globally in time for smooth initial data
near a constant state. This result indicated that the relaxation
effect could prevent the development of shock waves for the case of
smooth initial data with small oscillation. Hsiao and Wang \cite{HW}
considered the corresponding non-constant steady state solutions to
(\ref{R-E1}) and it was shown that the solutions were exponentially
locally asymptotically stable. For the Cauchy problem of
(\ref{R-E1}) with large smooth initial data, it was proved in
\cite{WC} the solution generally develops a singularity, shock
waves, and hence no global classical solution exists, which is due
to the strong hyperbolicity, even the damping relaxation and the
heat conduction (parabolicity) can't prevent the formation of
singularity. The effect of the Poisson coupling (ellipticity) is
smoothing and it decisively affects the stationary states of the
Euler-Poisson equations (\ref{R-E1}), see \cite{DM,Ga}.

Physically, it is more important and more interesting to study
(\ref{R-E1}) in several space dimensions where were expected to get
some similar results as the one-dimension case. Hsiao, Jiang and
Zhang \cite{HJZ} first studied the Cauchy-Neumann problem of
(\ref{R-E1}). Using the classical energy approach, they established
the global exponential stability of small smooth solutions near the
constant equilibrium. Subsequently, Li \cite{L} extended their
results to the non-constant equilibrium.

Recently, we started a program to investigate the hydrodynamic model
for semiconductor from the point of view of Fourier analysis, which
is more careful and refined manner. For instance, the method enables
us to understand the Poisson coupling effect well, which plays a key
role in the low frequency of density. Such a fact explains the
global exponential stability of small smooth solutions in essential.
Up to now, we have achieved some results in this direction, see
\cite{FXZ,X2,X3,XY1,XZ}. In these works, we focused on the case of
$\kappa\equiv0$ mainly and the Euler-Poisson equations (\ref{R-E1})
can be reduced to the pure hyperbolic form of the balance law with a
non-local source term by virtue of the Green's formulation. However,
the case of $\kappa\neq0$ is not the trivial one based on the
following considerations:
\begin{itemize}
\item[(1)]
The Euler-Poisson equations (\ref{R-E1}) has not scaling invariance,
pertinent to the compressible Navier-Stokes equations in
\cite{D1,D2}. Thus, we are going to choose the
\textit{non-homogeneous} Besov spaces $B^{N/p}_{p,1}(1\leq
p<\infty)$ as the basic functional setting (in $x$), which are the
critical spaces embedding in the space of Lipschitz functions;
\item[(2)]For the critical case of
regularity index, the classical existence theory for generally
hyperbolic systems established by Kato and Majda \cite{K,M} fails.
Iftimie \cite{I} first gave the contribution for general hyperbolic
systems, however, due to the complicated coupling, the result of
Iftimie can not be applied directly;
\item[(3)]The inequality in Proposition \ref{prop2.6} giving
parabolic regularity in the framework of Besov spaces depends on
given time $T$ (except for the case of $\alpha=1$ and
$\alpha_{1}=\infty$ ), which may preclude from proving global
existence results (even if small data).
\end{itemize}
Therefore, as the first step, we establish a local existence result
of Cauchy problem pertaining to data in the critical Besov spaces
for the Euler-Poisson equations (\ref{R-E1}) with $\kappa\neq0$. For
this purpose, the initial conditions for $n, \textbf{u}$ and
$\mathcal{T}$, and a boundary condition for $\mathit\Phi$ are
equipped:
\begin{eqnarray}
(n,\textbf{u},\mathcal{T})(x,0)=(n_{0}, \textbf{u}_{0},
\mathcal{T}_0)(x),\ \  x\in \mathbb{R}^{N}, \label{R-E2}
\end{eqnarray}
\begin{eqnarray}
\lim_{|x|\rightarrow +\infty}\mathit\Phi(t,x)=0, \ \ \ a.\ e. \ \
t>0, \label{R-E3}
\end{eqnarray}
where the homogeneous boundary condition for $\mathit\Phi$ means
that the semiconductor device is in equilibrium at infinity.

The local existence of a solution stems from the standard iterative
method. In comparison with that in \cite{D2}, there are some
differences in the proof of local existence. First, the estimate of
density $n$ does not follow from the estimates for the transport
equation in Besov spaces directly, since the velocity $\textbf{u}$
in the momentum equations has not higher regularity. Actually, to
obtain the desired frequency-localization estimate of $(n,
\textbf{u})$, we introduce a function change to reduce (\ref{R-E1})
to a part symmetric hyperbolic form, and take full advantage of
linear hyperbolic theory in the framework of Besov spaces and
hyperbolic energy method for dyadic blocks. Consequently, this
restricts us to the space case of $p=2$ in (1). Second, we show that
the approximate solution sequence is a Cauchy sequence in some norm
to prove the convergence rather than using compactness arguments. In
the meantime, in order to overcome the difficulty arising from the
heat conduction term, we develop a more general version of the
classical Moser-type inequality in Proposition \ref{prop2.3}, the
reader is refer to the Appendix. According to the new Moser-type
inequality, the heat conduction term can be estimated ultimately,
for details, see (\ref{R-E35})-(\ref{R-E36}). Finally, by a careful
analysis on the coupling effects, the uniqueness of classical
solutions is shown in the appropriately \textit{larger} spaces.

Through this paper, the regularity index $\sigma=1+N/2$. Our main
result is stated as follows.

\begin{thm} \label{thm1.1} Let \ $\bar{n},\mathcal{T}_{L}>0$ be the constant reference density and temperature.
Suppose that\ $n_{0}-\bar{n}, \mathbf{u}_{0},
\nabla\mathit\Phi(\cdot,0)\in B^{\sigma}_{2,1}(\mathbb{R}^{N})$ with
$n_{0}>0$ and $\mathcal{T}_{0}-\mathcal{T}_{L}\in
B^{\sigma+1}_{2,1}(\mathbb{R}^{N})$. Then there exists a time
$T_{1}>0$ such that
\begin{itemize}
\item[(i)] Existence:  the
system (\ref{R-E1})-(\ref{R-E3}) has a solution
$(n,\mathbf{u},\mathcal{T},\nabla\mathit\Phi)$ belongs to
$$(n,\mathbf{u},\mathcal{T}, \nabla\mathit\Phi)\in \mathcal{C}^{1}([0,T_{1}]\times \mathbb{R}^{N})\ \ \ \mbox{with}\
\ \  n>0\ \  \mbox{for all} \ \ t\in [0,T_{1}].$$ Furthermore, the
solution $(n,\mathbf{u},\mathcal{T},\nabla\mathit\Phi)$ satisfies
$$(n-\bar{n}, \mathbf{u}, \nabla\mathit\Phi)\in  \widetilde{\mathcal{C}}_{T_{1}}(B^{\sigma}_{2,1}(\mathbb{R}^{N}))
\times\Big(\widetilde{\mathcal{C}}_{T_{1}}(B^{\sigma}_{2,1}(\mathbb{R}^{N}))\Big)^{N}\times\Big(\widetilde{\mathcal{C}}_{T_{1}}(B^{\sigma}_{2,1}(\mathbb{R}^{N}))\Big)^{N}$$
and
$$\mathcal{T}-\mathcal{T}_{L}\in \widetilde{\mathcal{C}}_{T_{1}}(B^{\sigma+1}_{2,1}(\mathbb{R}^{N})).$$

\item[(ii)] Uniqueness: the solution
$(n,\mathbf{u},\mathcal{T},\nabla\mathit\Phi)$ is unique in the
spaces
$$\widetilde{L}^\infty_{T_{1}}(B^{\sigma-1}_{2,1})\times\Big(\widetilde{L}^\infty_{T_{1}}(B^{\sigma-1}_{2,1})\Big)^{N}
\times\widetilde{L}^1_{T_{1}}(B^{\sigma}_{2,1})\times\Big(\widetilde{L}^\infty_{T_{1}}(B^{\sigma-1}_{2,1})\Big)^{N},$$
\end{itemize}
 In
addition, there exists a constant $C_{0}>0$ depending only on
$\bar{n}, \mathcal{T}_{L}, N,\gamma$ such that
\begin{eqnarray}\|(n-\bar{n},\mathbf{u},\nabla\mathit\Phi)\|_{\widetilde{L}^\infty_{T_{1}}({B^{\sigma}_{2,1}})}
+\|\mathcal{T}-\mathcal{T}_{L}\|_{\widetilde{L}^\infty_{T_{1}}(B^{\sigma+1}_{2,1})}\leq
C_{0}M,\label{R-E1001}\end{eqnarray} where $M:=\|(n_{0}-\bar{n},
\mathbf{u}_{0},
\nabla\mathit\Phi(\cdot,0))\|_{B^{\sigma}_{2,1}}+\|\mathcal{T}_{0}-\mathcal{T}_{L}\|_{B^{\sigma+1}_{2,1}}$
and $\nabla\mathit\Phi(\cdot,0):=\nabla\Delta^{-1}(n_{0}-\bar{n})$.

\end{thm}

\begin{rem}\label{rem1.1}
 The symbol $\nabla\Delta^{-1}$ means
    $$\nabla\Delta^{-1}f=\int_{\mathbb{R}^{N}}\nabla_{x}G(x-y)f(y)dy,$$
 where $G(x,y)$ is a solution to $\Delta_xG(x,y)=\delta(x-y)$ with
 $x,y\in\mathbb{R}^N$.
\end{rem}

\begin{rem}\label{rem1.2}
Unlike \cite{D2},  \textit{no smallness condition} on the initial
density is required, hence the local well-posedness in critical
spaces holds for any initial density bounded away from zero. Let us
mention that the heat conductivity $\kappa$ in the proof is assumed
to be large appropriately, see the following inequalities
(\ref{R-E22}) and (\ref{R-E37}).
\end{rem}

\begin{rem}\label{rem1.3}
The local existence results in the framework of Sobolev spaces
$H^{s}$ with higher regularity ($s>1+N/2$) were obtained by the
standard contraction mapping principle, see, e.g., \cite{CJZ}.
Theorem \ref{thm1.1} deals with the limit case of regularity
($\sigma=1+N/2$), which is a natural generalization of their
results. To the best of our knowledge, the global well-posedness and
large-time behavior of (\ref{R-E1}) in critical spaces still remain
unsolved, since the inequality of parabolic regularity in
Proposition \ref{prop2.6} depends on the time $T$, which are under
current consideration.
\end{rem}

The rest of this paper unfolds as follows. In Section 2, we briefly
review the Littlewood-Paley decomposition theory and the
characterization of Besov spaces and Chemin-Lerner's spaces. Section
3 is dedicated to the proof of the local well-posedness of classical
solutions in critical spaces. Finally, the paper ends with an
appendix, where we prove a Moser-type inequality with aid of the
Bony's composition.

\textbf{Notations}. Throughout this paper, $C>0$ is a harmless
constant. Denote by $\mathcal{C}([0,T],X)$ (resp.,
$\mathcal{C}^{1}([0,T],X)$) the space of continuous (resp.,
continuously differentiable) functions on $[0,T]$ with values in a
Banach space $X$. For simplicity, the notation $\|(a,b,c,d)\|_{X}$
means $\|a\|_{X}+\|b\|_{X}+\|c\|_{X}+\|d\|_{X}$, where $a,b,c,d\in
X$. We shall omit the space dependence, since all functional spaces
are considered in $\mathbb{R}^{N}$. Moreover, the integral
$\int_{\mathbb{R}^{N}}fdx$ is labeled as $\int f$ without any
ambiguity.

\section{Tools}\setcounter{equation}{0}
The proofs of most of the results presented in this paper require a
dyadic decomposition of Fourier variable. Let us recall briefly the
Littlewood-Paley decomposition theory and the characterization of
Besov spaces and Chemin-Lerner's spaces, see for instance
\cite{BCD,D} for more details.

Let $(\varphi, \chi)$ be a couple of smooth functions valued in $[0,
1]$ such that $\varphi$ is supported in the shell
$\textbf{C}(0,\frac{3}{4},\frac{8}{3})
=\{\xi\in\mathbb{R}^{N}|\frac{3}{4}\leq|\xi|\leq\frac{8}{3}\}$,
$\chi$ is supported in the ball $\textbf{B}(0,\frac{4}{3})=
\{\xi\in\mathbb{R}^{N}||\xi|\leq\frac{4}{3}\}$ and
$$
\chi(\xi)+\sum_{q=0}^{\infty}\varphi(2^{-q}\xi)=1,\ \ \ \
\forall\xi\in\mathbb{R}^{N}.
$$
Let $\mathcal{S'}$ be the dual space of the Schwartz class
$\mathcal{S}$. For $f\in\mathcal{S'}$, the nonhomogeneous dyadic
blocks are defined as follows:
$$
\Delta_{-1}f:=\chi(D)f=\tilde{\omega}\ast f\ \ \ \mbox{with}\ \
\tilde{\omega}=\mathcal{F}^{-1}\chi;
$$
$$
\Delta_{q}f:=\varphi(2^{-q}D)f=2^{qd}\int \omega(2^{q}y)f(x-y)dy\ \
\ \mbox{with}\ \ \omega=\mathcal{F}^{-1}\varphi,\ \ \mbox{if}\ \
q\geq0,
$$
where $\ast$ the convolution operator and $\mathcal{F}^{-1}$ the
inverse Fourier transform. The nonhomogeneous Littlewood-Paley
decomposition is
$$
f=\sum_{q \geq-1}\Delta_{q}f \qquad \forall f\in \mathcal{S'}.
$$
Define the low frequency cut-off by
$$S_{q}f:=\sum_{p\leq q-1}\Delta_{p}f.$$ Of course, $S_{0}f=\Delta_{-1}f$. The above
Littlewood-Paley decomposition is almost orthogonal in $L^2$.
\begin{prop}\label{prop2.1}
For any $f\in\mathcal{S'}(\mathbb{R}^{N})$ and
$g\in\mathcal{S'}(\mathbb{R}^{d})$, the following properties hold:
$$\Delta_{p}\Delta_{q}f\equiv 0 \ \ \ \mbox{if}\ \ \ |p-q|\geq 2,$$
$$\Delta_{q}(S_{p-1}f\Delta_{p}g)\equiv 0\ \ \ \mbox{if}\ \ \ |p-q|\geq 5.$$
\end{prop}

Having defined the linear operators $\Delta_q (q\geq -1)$, we give
the definition of Besov spaces and Bony's decomposition.
\begin{defn}\label{defn2.1}
Let $1\leq p\leq\infty$ and $s\in \mathbb{R}$. For $1\leq r<\infty$,
Besov spaces $B^{s}_{p,r}\subset \mathcal{S'}$ are defined by
$$
f\in B^{s}_{p,r} \Leftrightarrow \|f\|_{B^s_{p, r}}=:
\Big(\sum_{q\geq-1}(2^{qs}\|\Delta_{q}f\|_{L^{p}})^{r}\Big)^{\frac{1}{r}}<\infty
$$
and $B^{s}_{p,\infty}\subset \mathcal{S'}$ are defined by
$$
f\in B^{s}_{p,\infty} \Leftrightarrow \|f\|_{B^s_{p, \infty}}=:
\sup_{q\geq-1}2^{qs}\|\Delta_{q}f\|_{L^{p}}<\infty.
$$
\end{defn}

\begin{defn}\label{defn2.2}
Let $f,g $ be two temperate distributions. The product $f\cdot g$
has the Bony's decomposition:
$$f\cdot g=T_{f}g+T_{g}f+R(f,g), $$
where $T_{f}g$ is paraproduct of $g$ by $f$,
$$ T_{f}g=\sum_{p\leq q-2}\Delta_{p}f\Delta_{q}g=\sum_{q}S_{q-1}f\Delta_{q}v$$
and the remainder $ R(f,g)$ is denoted by
$$R(f,g)=\sum_{q}\Delta_{q}f\tilde{\Delta}_{q}g\ \ \ \mbox{with} \ \
\tilde{\Delta}_{q}:=\Delta_{q-1}+\Delta_{q}+\Delta_{q+1}.$$
\end{defn}

As regards the remainder of para-product, we have the following
results.

\begin{prop} \label{prop2.2}
Let $(s_{1},s_{2})\in \mathbb{R}^2$ and $1\leq
p,p_{1},p_{2},r,r_{1},r_{2}\leq\infty$. Assume that
$$\frac{1}{p}\leq\frac{1}{p_{1}}+\frac{1}{p_{2}}\leq 1,\ \ \frac{1}{r}\leq\frac{1}{r_{1}}+\frac{1}{r_{2}},\ \ \mbox{and}\ \ s_{1}+s_{2}>0.$$
Then the remainder $R$ maps $B^{s_{1}}_{p_{1},r_{1}}\times
B^{s_{2}}_{p_{2},r_{2}}$ in
$B^{s_{1}+s_{2}+d(\frac{1}{p}-\frac{1}{p_{1}}-\frac{1}{p_{2}})}_{p,r}$
and there exists a constant $C$ such that
$$\|R(f,g)\|_{B^{s_{1}+s_{2}+d(\frac{1}{p}-\frac{1}{p_{1}}-\frac{1}{p_{2}})}_{p,r}}\leq \frac{C^{|s_{1}+s_{2}|+1}}{s_{1}+s_{2}}\|f\|_{B^{s_{1}}_{p_{1},r_{1}}}\|g\|_{B^{s_{2}}_{p_{2},r_{2}}}.$$
\end{prop}

Some conclusions will be used in subsequent analysis. The first one
is the classical Bernstein's inequality.

\begin{lem}\label{lem2.1}
Let $k\in\mathbb{N}$ and $0<R_{1}<R_{2}$. There exists a constant
$C$, depending only on $R_{1},R_{2}$ and $d$, such that for all
$1\leq a\leq b\leq\infty$ and $f\in L^{a}$,
$$
\mathrm{Supp}\ \mathcal{F}f\subset
\mathbf{B}(0,R_{1}\lambda)\Rightarrow\sup_{|\alpha|=k}\|\partial^{\alpha}f\|_{L^{b}}
\leq C^{k+1}\lambda^{k+d(\frac{1}{a}-\frac{1}{b})}\|f\|_{L^{a}};
$$
$$
\mathrm{Supp}\ \mathcal{F}f\subset
\mathbf{C}(0,R_{1}\lambda,R_{2}\lambda) \Rightarrow
C^{-k-1}\lambda^{k}\|f\|_{L^{a}}\leq
\sup_{|\alpha|=k}\|\partial^{\alpha}f\|_{L^{a}}\leq
C^{k+1}\lambda^{k}\|f\|_{L^{a}}.
$$
Here $\mathcal{F}f$ represents the Fourier transform on $f$.
\end{lem}

As a direct corollary of the above inequality, we have
\begin{rem}\label{rem2.1} For all
multi-index $\alpha$, it holds that
$$
\|\partial^\alpha f\|_{B^s_{p, r}}\leq C\|f\|_{B^{s + |\alpha|}_{p,
r}}.
$$
\end{rem}

The second one is the embedding properties in Besov spaces.
\begin{lem}\label{lem2.2} Let $s\in \mathbb{R}$ and $1\leq
p,r\leq\infty,$ then
$$B^{s}_{p,r}\hookrightarrow B^{\tilde{s}}_{p,\tilde{r}}\ \ \
\mbox{whenever}\ \ \tilde{s}<s\ \ \mbox{or}\ \ \tilde{s}=s \ \
\mbox{and}\ \ r\leq\tilde{r};$$
$$B^{s}_{p,r}\hookrightarrow B^{s-N(\frac{1}{p}-\frac{1}{\tilde{p}})}_{\tilde{p},r}\ \ \
\mbox{whenever}\ \ \tilde{p}>p;$$
$$B^{d/p}_{p,1}(1\leq p<\infty)\hookrightarrow\mathcal{C}_{0},\ \ \ B^{0}_{\infty,1}\hookrightarrow\mathcal{C}\cap L^{\infty},$$
where $\mathcal{C}_{0}$ is the space of continuous bounded functions
which decay at infinity.
\end{lem}

The third one is the traditional Moser-type inequality.

\begin{prop}\label{prop2.3}
Let $s>0$ and $1\leq p,r\leq\infty$. Then $B^{s}_{p,r}\cap
L^{\infty}$ is an algebra. Furthermore, it holds that
$$
\|fg\|_{B^{s}_{p,r}}\leq
C(\|f\|_{L^{\infty}}\|g\|_{B^{s}_{p,r}}+\|g\|_{L^{\infty}}\|f\|_{B^{s}_{p,r}}).
$$
\end{prop}

On the other hand, we present the definition of Chemin-Lerner's
spaces first introduced by J.-Y. Chemin and N. Lerner \cite{C},
which is the refinement of the spaces $L^{\rho}_{T}(B^{s}_{p,r})$.

\begin{defn}\label{defn2.3}
For $T>0, s\in\mathbb{R}, 1\leq r,\rho\leq\infty$, set (with the
usual convention if $r=\infty$)
$$\|f\|_{\widetilde{L}^{\rho}_{T}(B^{s}_{p,r})}:
=\Big(\sum_{q\geq-1}(2^{qs}\|\Delta_{q}f\|_{L^{\rho}_{T}(L^{p})})^{r}\Big)^{\frac{1}{r}}.$$
Then we define the space $\widetilde{L}^{\rho}_{T}(B^{s}_{p,r})$ as
the completion of $\mathcal{S}$ over $(0,T)\times\mathbb{R}^{d}$ by
the above norm.
\end{defn}
Furthermore, we define
$$\widetilde{\mathcal{C}}_{T}(B^{s}_{p,r}):=\widetilde{L}^{\infty}_{T}(B^{s}_{p,r})\cap\mathcal{C}([0,T],B^{s}_{p,r}), $$
where the index $T$ will be omitted when $T=+\infty$. Let us
emphasize that

\begin{rem}\label{rem2.2}
\rm According to Minkowski's inequality, it holds that
$$\|f\|_{\widetilde{L}^{\rho}_{T}(B^{s}_{p,r})}\leq\|f\|_{L^{\rho}_{T}(B^{s}_{p,r})}\,\,\,
\mbox{if}\,\, r\geq\rho;\ \ \ \
\|f\|_{\widetilde{L}^{\rho}_{T}(B^{s}_{p,r})}\geq\|f\|_{L^{\rho}_{T}(B^{s}_{p,r})}\,\,\,
\mbox{if}\,\, r\leq\rho.
$$\end{rem}
Then, we state the property of continuity for product in
Chemin-Lerner's spaces $\widetilde{L}^{\rho}_{T}(B^{s}_{p,r})$.
\begin{prop}\label{prop2.4}
The following estimate holds:
$$
\|fg\|_{\widetilde{L}^{\rho}_{T}(B^{s}_{p,r})}\leq
C(\|f\|_{L^{\rho_{1}}_{T}(L^{\infty})}\|g\|_{\widetilde{L}^{\rho_{2}}_{T}(B^{s}_{p,r})}
+\|g\|_{L^{\rho_{3}}_{T}(L^{\infty})}\|f\|_{\widetilde{L}^{\rho_{4}}_{T}(B^{s}_{p,r})})
$$
whenever $s>0, 1\leq p\leq\infty,
1\leq\rho,\rho_{1},\rho_{2},\rho_{3},\rho_{4}\leq\infty$ and
$$\frac{1}{\rho}=\frac{1}{\rho_{1}}+\frac{1}{\rho_{2}}=\frac{1}{\rho_{3}}+\frac{1}{\rho_{4}}.$$
As a direct corollary, one has
$$\|fg\|_{\widetilde{L}^{\rho}_{T}(B^{s}_{p,r})}
\leq
C\|f\|_{\widetilde{L}^{\rho_{1}}_{T}(B^{s}_{p,r})}\|g\|_{\widetilde{L}^{\rho_{2}}_{T}(B^{s}_{p,r})}$$
whenever $s\geq N/p,
\frac{1}{\rho}=\frac{1}{\rho_{1}}+\frac{1}{\rho_{2}}.$
\end{prop}

In addition, the estimates of commutators in
$\widetilde{L}^{\rho}_{T}(B^{s}_{p,1})$ spaces are also frequently
used in the subsequent analysis. The indices $s,p$ behave just as in
the stationary case \cite{D,FXZ} whereas the time exponent $\rho$
behaves according to H\"{o}lder inequality.
\begin{lem}\label{lem2.3}
Let $1\leq p<\infty$ and $1\leq \rho\leq\infty$,  then the following
inequalities are true:
\begin{eqnarray*}
&&2^{qs}\|[f,\Delta_{q}]\mathcal{A}g\|_{L^{\rho}_{T}(L^{p})}\nonumber\\&\leq&
\left\{
\begin{array}{l}
 Cc_{q}\|f\|_{\widetilde{L}^{\rho_{1}}_{T}(B^{s}_{p,1})}\|g\|_{\widetilde{L}^{\rho_{2}}_{T}(B^{s}_{p,1})},\
\ s=1+N/p,\\
 Cc_{q}\|f\|_{\widetilde{L}^{\rho_{1}}_{T}(B^{s}_{p,1})}\|g\|_{\widetilde{L}^{\rho_{2}}_{T}(B^{s+1}_{p,1})},\
 \ s=N/p,\\
 Cc_{q}\|f\|_{\widetilde{L}^{\rho_{1}}_{T}(B^{s+1}_{p,1})}\|g\|_{\widetilde{L}^{\rho_{2}}_{T}(B^{s}_{p,1})},\ \
s=N/p,
\end{array} \right.
\end{eqnarray*}
where the commutator $[\cdot,\cdot]$ is defined by $[f,g]=fg-gf$,
the operator $\mathcal{A}=\mathrm{div}$ or $\mathrm{\nabla}$, $C$ is
a generic constant, and $c_{q}$ denotes a sequence such that
$\|(c_{q})\|_{ {l^{1}}}\leq
1,\frac{1}{\rho}=\frac{1}{\rho_{1}}+\frac{1}{\rho_{2}}.$
\end{lem}

In the symmetrization, we shall face with some composition
functions. To estimate them, the following
 continuity result for compositions is necessary.
\begin{prop}\label{prop2.5}
Let $s>0$, $1\leq p, r, \rho\leq \infty$, $F\in
W^{[s]+1,\infty}_{loc}(I;\mathbb{R})$ with $F(0)=0$, $T\in
(0,\infty]$ and $v\in \widetilde{L}^{\rho}_{T}(B^{s}_{p,r})\cap
L^{\infty}_{T}(L^{\infty}).$ Then
$$\|F(v)\|_{\widetilde{L}^{\rho}_{T}(B^{s}_{p,r})}\leq
C(1+\|v\|_{L^{\infty}_{T}(L^{\infty})})^{[s]+1}\|v\|_{\widetilde{L}^{\rho}_{T}(B^{s}_{p,r})}.$$
\end{prop}

Finally, we give the estimate of heat equation to end up this
section.
\begin{prop}\label{prop2.6}
Let $s\in \mathbb{R}$ and $1\leq\alpha,p,r\leq\infty$. Let $T>0,
u_{0}\in B^{s}_{p,r}$ and $f\in
\widetilde{L}^{\alpha}_{T}(B^{s-2+\frac{2}{\alpha}}_{p,r})$. Then
the problem of heat equation
$$\partial_{t}u-\mu\Delta u=f,\ \ \ u|_{t=0}=u_{0}$$
has a unique solution $u\in
\widetilde{L}^{\alpha}_{T}(B^{s+\frac{2}{\alpha}}_{p,r})\cap\widetilde{L}^{\infty}_{T}(B^{s}_{p,r})
$ and there exists a constant $C$ depending only on $N$ and such
that for all $\alpha_{1}\in[\alpha,+\infty]$, we have
$$\mu^{\frac{1}{\alpha_{1}}}\|u\|_{\widetilde{L}^{\alpha_{1}}_{T}(B^{s+\frac{2}{\alpha}}_{p,r})}\leq C\Big\{(1+T^{\frac{1}{\alpha_{1}}})\|u_{0}\|_{B^{s}_{p,r}}
+(1+T^{1+\frac{1}{\alpha_{1}}-\frac{1}{\alpha}})\mu^{\frac{1}{\alpha}-1}\|f\|_{\widetilde{L}^{\alpha}_{T}(B^{s-2+\frac{2}{\alpha}}_{p,r})}\Big\}.$$
In addition, if $r$ is finite then $u$ belongs to
$\mathcal{C}([0,T];B^{s}_{p,r})$.
\end{prop}

\section{Well-posedness for $\kappa\neq0$ }
\setcounter{equation}{0} In this section, using the
frequency-localization methods, we give the proof of main result.

\noindent\textit{\underline{The proof of Theorem \ref{thm1.1}.}} The
coefficients $\tau_{p},\tau_{w},\lambda$ are assumed to be one. For
classical solutions, (\ref{R-E1}) can be changed into the following
system in $(n,\textbf{u},\mathcal{T},\mathit\Phi)$:
\begin{equation}
\left\{
\begin{array}{l}\partial_{t}n+\mathrm{div}(n\textbf{u})=0,\\
n\partial_{t}\textbf{u}+n(\textbf{u}\cdot\nabla)\textbf{u}+\nabla
(n\mathcal{T})
 =n\nabla\mathit\Phi-n\textbf{u},
 \\
n\partial_{t}\mathcal{T}+n\textbf{u}\cdot\nabla
\mathcal{T}+(\gamma-1)n\mathcal{T}\mathrm{div}
\textbf{u}=(\gamma-1)\kappa\Delta \mathcal{T}
+\frac{\gamma-1}{2}n|\textbf{u}|^{2}-n(\mathcal{T}-\mathcal{T}_{L}),\\
\Delta\mathit\Phi=n-\overline{n},
 \end{array} \right.\label{R-E4}
\end{equation}
In order to obtain the effective frequency-localization estimate on
$(n,\textbf{u})$, we introduce a function change
$$\left(%
\begin{array}{c}
  \rho\\
  \textbf{u}\\
 \theta \\
  \textbf{E}\\
\end{array}%
\right)=\left(%
\begin{array}{c}
  \ln n-\ln \overline{n} \\
  \textbf{u} \\
  \mathcal{T}-\mathcal{T}_{L} \\
  \nabla\mathit\Phi  \\
\end{array}%
\right).
$$
Then the new variable $(\rho,\textbf{u},\theta,\textbf{E})$
satisfies
\begin{equation}
\left\{
\begin{array}{l}\partial_{t}\rho+\textbf{u}\cdot\nabla \rho+\mathrm{div}\textbf{u}=0,\\
\partial_{t}\textbf{u}+\mathcal{T}_{L}\nabla \rho+(\textbf{u}\cdot\nabla)\textbf{u}+\nabla \theta
+\theta\nabla \rho=\textbf{E}-\textbf{u}, \\
\partial_{t}\theta-\frac{(\gamma-1)\kappa}{\bar{n}}\Delta\theta+\textbf{u}\cdot\nabla
\theta=h_{1}(\rho)\Delta\theta-(\gamma-1)(\mathcal{T}_{L}+\theta)
\mathrm{div}\textbf{u}
+\frac{\gamma-1}{2}|\textbf{u}|^{2}-\theta,\\
\partial_{t}\textbf{E}=-\nabla\Delta^{-1}\mathrm{div}(h_{2}(\rho)\textbf{u}+\bar{n}\textbf{u}),
 \end{array} \right.\label{R-E5}
\end{equation}
where $h_{1}(\rho), h_{2}(\rho) $ defined by
$$h_{1}(\rho)=\frac{(\gamma-1)\kappa}{\bar{n}}\Big(1-\exp(-\rho)\Big)\ \mbox{and} \ \ h_{2}(\rho)=\bar{n}(\exp(\rho)-1)$$
are two smooth functions on the interval $(-\infty,\infty)$. The
non-local term $\nabla\Delta^{-1}\nabla \cdot f$ is the product of
Riesz transforms on $f$. Here and below, we set
$\tilde{\kappa}=\frac{(\gamma-1)\kappa}{\bar{n}}$ for simplicity.

The initial data (\ref{R-E2}) become
\begin{equation}(\rho,\textbf{u},\theta,\textbf{E})(x,0)
=(\ln n_{0}-\ln
\bar{n},\textbf{u}_{0},\mathcal{T}_{0}-\mathcal{T}_{L},\nabla\Delta^{-1}(n_{0}-\bar{n})),\
x\in \mathbb{R}^{N}. \label{R-E6}\end{equation}

\begin{rem}\label{rem3.1}
The variable transform is from the open set
$\{(n,\textbf{u},\mathcal{T},\textbf{E})\in (0,+\infty)\times
\mathbb{R}^{N}\times\mathbb{R}^{N}\times \mathbf{R}^{N}\}$ to the
whole space $\{(\rho,\textbf{u},\theta,\textbf{E})\in
\mathbb{R}\times \mathbb{R}^{N}\times\mathbb{R}\times
\mathbb{R}^{N}\}$. It is easy to show that for classical solutions
$(n,\textbf{u},\mathcal{T},\textbf{E})$ away from vacuum,
(\ref{R-E1})-(\ref{R-E2}) is equivalent to
(\ref{R-E5})-(\ref{R-E6}).
\end{rem}

The proof of the local well-posedness stems from a standard
iterative process. First of all, we consider the linear coupled
system of hyperbolic-parabolic form
\begin{equation}
\left\{
\begin{array}{l}\partial_{t}\overline{\rho}+\textbf{v}\cdot\nabla
\overline{\rho}+\mathrm{div}\overline{\textbf{u}}=0,\\[1mm]
\partial_{t}\overline{\textbf{u}}+\mathcal{T}_{L}\nabla \overline{\rho}+(\textbf{v}\cdot\nabla)\overline{\textbf{u}}=f, \\[1mm]
\partial_{t}\overline{\theta}-\tilde{\kappa}\Delta\overline{\theta}=g,\\[1mm]
\partial_{t}\overline{\textbf{E}}=-\nabla\Delta^{-1}\mathrm{div}h,
 \end{array} \right.\label{R-E1002}
\end{equation}
subject to the initial data
\begin{eqnarray}
(\overline{\rho},\overline{\textbf{u}},\overline{\theta},\overline{\textbf{E}})|_{t=0}=(\overline{\rho}_{0},\overline{\mathbf{u}}_{0},\overline{\theta}_{0},\overline{\mathbf{E}}_{0}),
\label{R-E1003}
\end{eqnarray}
where $\textbf{v},f,h:\mathbb{R}^{+}\times \mathbb{R}^{N}\rightarrow
\mathbb{R}^{N}$ and $g:\mathbb{R}^{+}\times
\mathbb{R}^{N}\rightarrow \mathbb{R}$.

For the system (\ref{R-E1002})-(\ref{R-E1003}), we have the
following conclusion.
\begin{prop}\label{prop3.1}
Let $p\in[1,+\infty],\ r\in[1,+\infty),\ s_{1}>0$  and
$s_{2}\in\mathbb{R}$. Suppose that
$(\overline{\rho}_{0},\overline{\mathbf{u}}_{0},\overline{\mathbf{E}}_{0})\in
B^{s_{1}}_{2,r},\ \overline{\theta}_{0}\in B^{s_{2}}_{p,r}$, $f,h\in
\mathcal{C}([0,T], B^{s_{1}}_{2,r}),\ g\in L^1(0,T;B^{s_{2}}_{p,r})$
and
\begin{eqnarray*}
&&\nabla\mathbf{v}\in \left\{
\begin{array}{l}
 \mathcal{C}([0,T], B^{s_{1}-1}_{2,r})\ \ \mbox{if}\ \ s_{1}>1+N/2, \mbox{or}\ s_{1}=1+N/2 \ \mbox{and}\ r=1;\\
 \mathcal{C}([0,T], B^{\frac{N}{2}+\varepsilon}_{2,\infty})\ \ \mbox{for some}\ \ \varepsilon>0\ \ \mbox{if}\ \ s_{1}=1+N/2 \ \mbox{and}\ r>1;\\
 \mathcal{C}([0,T], B^{\frac{N}{2}}_{2,\infty}\cap L^{\infty}) \ \ \mbox{if}\ \ 0<s_{1}<1+N/2;\\
\end{array} \right.
\end{eqnarray*}
for any given $T>0$. Then the system (\ref{R-E1002})-(\ref{R-E1003})
has a unique solution
$(\overline{\rho},\overline{\mathbf{u}},\overline{\theta},\overline{\mathbf{E}})$
satisfying
$$(\overline{\rho},\overline{\mathbf{u}},\overline{\mathbf{E}})\in\widetilde{\mathcal{C}}_{T}(B^{s_{1}}_{2,r})\ \ \mbox{and} \ \ \overline{\theta}\in\widetilde{\mathcal{C}}_{T}(B^{s_{2}}_{p,r}).$$
\end{prop}

\begin{proof}
Note that the $L^2$- boundedness of Riesz transform, Proposition
\ref{prop3.1} is the direct consequence of Proposition \ref{prop2.6}
and Theorem 4.15 in the recent book \cite{BCD}.
\end{proof}

In what follows, the proof of Theorem \ref{thm1.1} is divided into
several steps, since it is a bit longer.\\

\textbf{Step1: approximate solutions}

We use a standard iterative process to build a solution. Starting
from $(\rho^0,\textbf{u}^0,\theta^0,\textbf{E}^0):=(0,0,0,0)$. Then
we define by induction a  solution sequence
$\{(\rho^m,\textbf{u}^m,\theta^m,\textbf{E}^m)\}_{m\in \mathbb{N}}$
by solving the following linear equations
\begin{equation}
\left\{
\begin{array}{l}\partial_{t}\rho^{m+1}+\textbf{u}^{m}\cdot\nabla
\rho^{m+1}+\mathrm{div}\textbf{u}^{m+1}=0,\\[1mm]
\partial_{t}\textbf{u}^{m+1}+\mathcal{T}_{L}\nabla \rho^{m+1}+(\textbf{u}^{m}\cdot\nabla)\textbf{u}^{m+1}=-\nabla \theta^{m}
-\theta^{m}\nabla \rho^{m}+\textbf{E}^{m}-\textbf{u}^{m}, \\[1mm]
\partial_{t}\theta^{m+1}-\tilde{\kappa}\Delta\theta^{m+1}\\ \hspace{10mm}=-\textbf{u}^{m}\cdot\nabla
\theta^{m}+h_{1}(\rho^{m})\Delta\theta^{m}-(\gamma-1)(\mathcal{T}_{L}+\theta^{m})
\mathrm{div}\textbf{u}^{m}
+\frac{\gamma-1}{2}|\textbf{u}^{m}|^{2}-\theta^{m},\\[1mm]
\partial_{t}\textbf{E}^{m+1}=-\nabla\Delta^{-1}\mathrm{div}\{h_{2}(\rho^m)\textbf{u}^{m}+\bar{n}\textbf{u}^{m}\},
 \end{array} \right.\label{R-E7}
\end{equation}
with the initial data
\begin{equation}
(\rho^{m+1},\textbf{u}^{m+1},\theta^{m+1},\textbf{E}^{m+1})(x,0)=(S_{m+1}\rho_{0},S_{m+1}\textbf{u}_{0},S_{m+1}\theta_{0},S_{m+1}\textbf{E}_{0}),\
x\in \mathbb{R}^{N}. \label{R-E8}\end{equation}

Since all the data belong to $B^{\infty}_{2,r}$, Proposition
\ref{prop3.1} enable us to show by induction that the above Cauchy
problem has a global solution which belongs to
$\widetilde{\mathcal{C}}(B^{\infty}_{2,r})$.\\

\textbf{Step2: uniform bounds}

Set
$$E_{T}^{\sigma}:=\widetilde{\mathcal{C}}_{T}(B^{\sigma}_{2,1})\times\Big(\widetilde{\mathcal{C}}_{T}(B^{\sigma}_{2,1})\Big)^{N}
\times\widetilde{\mathcal{C}}_{T}(B^{\sigma+1}_{2,1})\times\Big(\widetilde{\mathcal{C}}_{T}(B^{\sigma}_{2,1}))\Big)^{N}$$
for $T>0$. We hope to find a time $T$ such that the approximate
solution
$\{(\rho^{m},\textbf{u}^{m},\theta^{m},\textbf{E}^{m})\}_{m\in
\mathbb{N}}$ is uniformly bounded in $E_{T}^{\sigma}$.

First, by applying the operator $\Delta_{q}(q\geq-1)$ to the first
two equations of (\ref{R-E7}), we infer that for
$(\Delta_{q}\rho^{m+1},\Delta_{q}\textbf{u}^{m+1})$
\begin{equation}
\left\{
\begin{array}{l}\partial_{t}\Delta_{q}\rho^{m+1}+(\textbf{u}^{m}\cdot\nabla)\Delta_{q}\rho^{m+1}+\Delta_{q}\mathrm{div}\textbf{u}^{m+1}=[\textbf{u}^{m},\Delta_{q}]\cdot\nabla\rho^{m+1},\\[2mm]
\partial_{t}\Delta_{q}\textbf{u}^{m+1}+\mathcal{T}_{L}\Delta_{q}\nabla \rho^{m+1}+(\textbf{u}^{m}\cdot\nabla)\Delta_{q}\textbf{u}^{m+1}\\
\hspace{5mm}=-\Delta_{q}\nabla \theta^{m}
+[\textbf{u}^{m},\Delta_{q}]\cdot\nabla\textbf{u}^{m+1}-\nabla
\rho^{m}\Delta_{q}\theta^{m}+[\nabla
\rho^{m},\Delta_{q}]\theta^{m}+\Delta_{q}\textbf{E}^{m}-\Delta_{q}\textbf{u}^{m},
\end{array} \right.\label{R-E9}
\end{equation}
where the commutator $[\cdot,\cdot]$ is defined by $[f,g]=fg-gf$.

Then multiplying the first equation of Eqs. (\ref{R-E9}) by
$\mathcal{T}_{L}\Delta_{q}\rho^{m+1}$, the second one by
$\Delta_{q}\textbf{u}^{m+1}$, and adding the resulting equations
together, after integrating it over $\mathbb{R}^{N}$, we have
\begin{eqnarray}&&\frac{1}{2}\frac{d}{dt}\Big(\mathcal{T}_{L}\|\Delta_{q}\rho^{m+1}\|^2_{L^2}+\|\Delta_{q}\textbf{u}^{m+1}\|^2_{L^2}\Big)\nonumber
\\&=&\frac{1}{2}\int\mathrm{div}\textbf{u}^{m}(\mathcal{T}_{L}|\Delta_{q}\rho^{m+1}|^2+|\Delta_{q}\textbf{u}^{m+1}|^2)+\int\mathcal{T}_{L}[\textbf{u}^{m},\Delta_{q}]\cdot\nabla
\rho^{m+1}\Delta_{q}\rho^{m+1}\nonumber\\&&-\int\Delta_{q}\nabla\theta^{m}\cdot\Delta_{q}\textbf{u}^{m+1}
+\int[\textbf{u}^{m},\Delta_{q}]\cdot\nabla\textbf{u}^{m+1}\Delta_{q}\textbf{u}^{m+1}-\int\nabla\rho^{m}\cdot\Delta_{q}\textbf{u}^{m+1}\Delta_{q}\theta^{m}
\nonumber\\&&+\int[\nabla
\rho^{m},\Delta_{q}]\theta^{m}\cdot\Delta_{q}\textbf{u}^{m+1}+\int\Delta_{q}\textbf{E}^{m}\cdot\Delta_{q}\textbf{u}^{m+1}
-\int\Delta_{q}\textbf{u}^{m}\cdot\Delta_{q}\textbf{u}^{m+1}\nonumber
\\&\leq&
\frac{1}{2}\|\nabla\textbf{u}^{m}\|_{L^\infty}(\|\Delta_{q}\rho^{m+1}\|^2_{L^2}+\|\Delta_{q}\textbf{u}^{m+1}\|^2_{L^2})+\mathcal{T}_{L}\|[\textbf{u}^{m},\Delta_{q}]\cdot\nabla
\rho^{m+1}\|_{L^2}\|\Delta_{q}\rho^{m+1}\|_{L^2}\nonumber\\&&+\|\Delta_{q}\nabla\theta^{m}\|_{L^2}\|\Delta_{q}\textbf{u}^{m+1}\|_{L^2}+
\|[\textbf{u}^{m},\Delta_{q}]\cdot\nabla
\textbf{u}^{m+1}\|_{L^2}\|\Delta_{q}\textbf{u}^{m+1}\|_{L^2}\nonumber\\&&+\|\nabla\rho^{m}\|_{L^\infty}\|\Delta_{q}\theta^{m}\|_{L^2}\|\Delta_{q}\textbf{u}^{m+1}\|_{L^2}
+\|[\nabla
\rho^{m},\Delta_{q}]\theta^{m}\|_{L^2}\|\Delta_{q}\textbf{u}^{m+1}\|_{L^2}\nonumber\\&&+\|\Delta_{q}\textbf{E}^{m}\|_{L^2}\|\Delta_{q}\textbf{u}^{m+1}\|_{L^2}
+\|\Delta_{q}\textbf{u}^{m}\|_{L^2}\|\Delta_{q}\textbf{u}^{m+1}\|_{L^2},
\label{R-E10}\end{eqnarray} where we have used Cauchy-Schwartz's
inequality.

Dividing (\ref{R-E10}) by
$\Big(\mathcal{T}_{L}\|\Delta_{q}\rho^{m+1}\|^2_{L^2}+\|\Delta_{q}\textbf{u}^{m+1}\|^2_{L^2}+\varepsilon\Big)^{\frac{1}{2}}$
($\varepsilon>0$ is a small quantity), we get
\begin{eqnarray}&&\frac{d}{dt}\Big(\mathcal{T}_{L}\|\Delta_{q}\rho^{m+1}\|^2_{L^2}+\|\Delta_{q}\textbf{u}^{m+1}\|^2_{L^2}+\varepsilon\Big)^{\frac{1}{2}}\nonumber
\\&\leq&
C\|\nabla\textbf{u}^{m}\|_{L^\infty}(\|\Delta_{q}\rho^{m+1}\|_{L^2}+\|\Delta_{q}\textbf{u}^{m+1}\|_{L^2})+C\|[\textbf{u}^{m},\Delta_{q}]\cdot\nabla
\rho^{m+1}\|_{L^2}\nonumber\\&&+C\|\Delta_{q}\nabla\theta^{m}\|_{L^2}+C
\|[\textbf{u}^{m},\Delta_{q}]\cdot\nabla
\textbf{u}^{m+1}\|_{L^2}+C\|\nabla\rho^{m}\|_{L^\infty}\|\Delta_{q}\theta^{m}\|_{L^2}
\nonumber\\&&+C\|[\nabla
\rho^{m},\Delta_{q}]\theta^{m}\|_{L^2}+C\|\Delta_{q}\textbf{E}^{m}\|_{L^2}
+C\|\Delta_{q}\textbf{u}^{m}\|_{L^2}, \label{R-E11}\end{eqnarray}
where $C>0$ here and below denotes a uniform constant independent of
$m$. Integrating (\ref{R-E11}) with respect to the variable
$t\in[0,T]$, then taking $\varepsilon\rightarrow0$, we arrive at
\begin{eqnarray}&&\|\Delta_{q}\rho^{m+1}(t)\|_{L^2}+\|\Delta_{q}\textbf{u}^{m+1}(t)\|_{L^2}\nonumber
\\&\leq&C(\|\Delta_{q}\rho^{m+1}_{0}\|_{L^2}+\|\Delta_{q}\textbf{u}^{m+1}_{0}\|_{L^2})+C\int^{t}_{0}
\|\nabla\textbf{u}^{m}(\tau)\|_{L^\infty}\Big(\|\Delta_{q}\rho^{m+1}(\tau)\|_{L^2}\nonumber\\&&+\|\Delta_{q}\textbf{u}^{m+1}(\tau)\|_{L^2}\Big)d\tau
+C\int^{t}_{0}\Big(\|[\textbf{u}^{m},\Delta_{q}]\cdot\nabla
\rho^{m+1}\|_{L^2}+\|[\textbf{u}^{m},\Delta_{q}]\cdot\nabla
\textbf{u}^{m+1}\|_{L^2}\Big)d\tau\nonumber\\&&+C\int^{t}_{0}\Big(\|\Delta_{q}\nabla\theta^{m}\|_{L^2}
+\|\nabla\rho^{m}\|_{L^\infty}\|\Delta_{q}\theta^{m}\|_{L^2}
+\|[\nabla
\rho^{m},\Delta_{q}]\theta^{m}\|_{L^2}\nonumber\\&&+\|\Delta_{q}\textbf{E}^{m}\|_{L^2}
+\|\Delta_{q}\textbf{u}^{m}\|_{L^2}\Big)d\tau.
\label{R-E12}\end{eqnarray} Multiply the factor
$2^{q\sigma}(\sigma=1+N/2)$ on both sides of (\ref{R-E12}) to obtain
\begin{eqnarray}&&2^{q\sigma}\|\Delta_{q}\rho^{m+1}(t)\|_{L^2}+2^{q\sigma}\|\Delta_{q}\textbf{u}^{m+1}(t)\|_{L^2}\nonumber
\\&\leq&C2^{q\sigma}(\|\Delta_{q}\rho^{m+1}_{0}\|_{L^2}+\|\Delta_{q}\textbf{u}^{m+1}_{0}\|_{L^2})+C\int^{t}_{0}
\|\textbf{u}^{m}(\tau)\|_{B^{\sigma}_{2,1}}2^{q\sigma}\Big(\|\Delta_{q}\rho^{m+1}(\tau)\|_{L^2}\nonumber\\&&+\|\Delta_{q}\textbf{u}^{m+1}(\tau)\|_{L^2}\Big)d\tau
+C\int^{t}_{0}c_{q}(\tau)\|\textbf{u}^{m}\|_{B^{\sigma}_{2,1}}\Big(\|\rho^{m+1}(\tau)\|_{B^{\sigma}_{2,1}}+\|
\textbf{u}^{m+1}(\tau)\|_{B^{\sigma}_{2,1}}\Big)d\tau\nonumber\\&&+C\int^{t}_{0}c_{q}(\tau)\|\rho^{m}\|_{B^{\sigma}_{2,1}}\|\theta^{m}\|_{B^{\sigma}_{2,1}}d\tau
+C\int^{t}_{0}2^{q\sigma}\Big(\|\Delta_{q}\nabla\theta^{m}\|_{L^2}
\nonumber\\&&+\|\nabla\rho^{m}\|_{L^\infty}\|\Delta_{q}\theta^{m}\|_{L^2}
+\|\Delta_{q}\textbf{E}^{m}\|_{L^2}
+\|\Delta_{q}\textbf{u}^{m}\|_{L^2}\Big)d\tau
\label{R-E13}\end{eqnarray} where we used Remark \ref{rem2.1} and
Lemmas \ref{lem2.2}-\ref{lem2.3} and $\{c_{q}\}$ denotes some
sequence which satisfies $\|(c_{q})\|_{ {l^{1}}}\leq 1$ although
each $\{c_{q}\}$ is possibly different in (\ref{R-E13}).

Summing up (\ref{R-E13}) on $q\geq-1$ implies
\begin{eqnarray}&&\|(\rho^{m+1},\textbf{u}^{m+1})\|_{\widetilde{L}^{\infty}_{T}(B^{\sigma}_{2,1})}\nonumber
\\&\leq&C(\|(\rho^{m+1}_{0},\textbf{u}^{m+1}_{0})\|_{B^{\sigma}_{2,1}}+C\int^{T}_{0}
\|\textbf{u}^{m}(t)\|_{B^{\sigma}_{2,1}}\|(\rho^{m+1},\textbf{u}^{m+1})(t)\|_{\widetilde{L}^{\infty}_{t}(B^{\sigma}_{2,1})}dt\nonumber\\&&
\nonumber\\&&+C\int^{T}_{0}\Big((1+\|\rho^{m}\|_{B^{\sigma}_{2,1}})\|\theta^{m}\|_{B^{\sigma+1}_{2,1}}+
\|(\textbf{u}^{m},\textbf{E}^{m})\|_{B^{\sigma}_{2,1}}\Big)dt.
\label{R-E14}\end{eqnarray} Then it follows from Gronwall's
inequality that
\begin{eqnarray}&&\|(\rho^{m+1},\textbf{u}^{m+1})\|_{\widetilde{L}^{\infty}_{T}(B^{\sigma}_{2,1})}\nonumber
\\&\leq&Ce^{CZ^{m}(T)}\Big\{\|(\rho_{0},\textbf{u}_{0})\|_{B^{\sigma}_{2,1}}\nonumber\\&&+\int^{T}_{0}e^{-CZ^{m}(t)}
\Big((1+\|\rho^{m}(t)\|_{B^{\sigma}_{2,1}})\|\theta^{m}(t)\|_{B^{\sigma+1}_{2,1}}+
\|(\textbf{u}^{m},\textbf{E}^{m})(t)\|_{B^{\sigma}_{2,1}}\Big)dt\Big\},
\label{R-E15}\end{eqnarray} with
$Z^{m}(T):=\int^T_{0}\|\textbf{u}^{m}(t)\|_{B^{\sigma}_{2,1}}dt.$

On the other hand, by the last equation of (\ref{R-E7}), with the
aid of Proposition \ref{prop2.5}, we can obtain
\begin{eqnarray}
&&\|\textbf{E}^{m+1}\|_{\widetilde{L}^{\infty}_{T}(B^{\sigma}_{2,1})}\nonumber\\&\leq&
C\Big(\|\textbf{E}^{m+1}_{0}\|_{B^{\sigma}_{2,1}}+\int^{T}_{0}\|\nabla\Delta^{-1}\mathrm{div}(h_{2}(\rho^m)\textbf{u}^m+\bar{n}\textbf{u}^m)\|_{B^{\sigma}_{2,1}}dt\Big)
\nonumber\\&\leq&C\Big(\|\textbf{E}_{0}\|_{B^{\sigma}_{2,1}}+\int^{T}_{0}(1+\|\rho^m\|_{B^{\sigma}_{2,1}})\|\textbf{u}^m\|_{B^{\sigma}_{2,1}}dt\Big),
\label{R-E16}\end{eqnarray} where we have used the $L^2$-boundedness
of nonlocal (but zero order) operator
$\nabla\Delta^{-1}\mathrm{div}$.

Taking $\alpha_{1}=\alpha=\infty,\ s=\sigma+1, \ p=2$ and $r=1$ in
Proposition \ref{prop2.6}, and applying the resulting inequality to
the third equation of (\ref{R-E7}), we have
\begin{eqnarray}
\|\theta^{m+1}\|_{\widetilde{L}^\infty_{T}(B^{\sigma+1}_{2,1})}\leq
C\Big(\|\theta^{m+1}_{0}\|_{B^{\sigma+1}_{2,1}}+(1+T)\tilde{\kappa}^{-1}\|F_{1}^{m}\|_{\widetilde{L}^\infty_{T}(B^{\sigma-1}_{2,1})}\Big),\label{R-E17}
\end{eqnarray}
where $F_{1}^{m}:=-\textbf{u}^{m}\cdot\nabla
\theta^{m}+h_{1}(\rho^{m})\Delta\theta^{m}-(\gamma-1)(\mathcal{T}_{L}+\theta^{m})
\mathrm{div}\textbf{u}^{m}
+\frac{\gamma-1}{2}|\textbf{u}^{m}|^{2}-\theta^{m}.$ From
Propositions \ref{prop2.4}-\ref{prop2.5} and Lemma \ref{lem2.2}, we
are led to
\begin{eqnarray}
&&\|\theta^{m+1}\|_{\widetilde{L}^\infty_{T}(B^{\sigma+1}_{2,1})}\nonumber\\&\leq&
C\Big\{\|\theta_{0}\|_{B^{\sigma+1}_{2,1}}+(1+T)\tilde{\kappa}^{-1}\Big((1+\|(\rho^{m},\textbf{u}^{m})\|_{\widetilde{L}^\infty_{T}(B^{\sigma}_{2,1})})\|\theta^{m}\|_{\widetilde{L}^\infty_{T}(B^{\sigma+1}_{2,1})}
\nonumber\\&&+(1+\|\textbf{u}^{m}\|_{\widetilde{L}^\infty_{T}(B^{\sigma}_{2,1})})\|\textbf{u}^{m}\|_{\widetilde{L}^\infty_{T}(B^{\sigma}_{2,1})}\Big)\Big\}.\label{R-E18}
\end{eqnarray}
Note that although the above constant $C$ maybe depend on $N$, it is
nothing to do with $m$, so we obtain the following uniform estimates. \\

\begin{lem}
 There exists a time $T_{1}>0$
(independent of $m$) such that
\begin{eqnarray}\|(\rho^{m},\mathbf{u}^{m},\mathbf{E}^{m})\|_{\widetilde{L}^\infty_{T_{1}}({B^{\sigma}_{2,1}})}+\|\theta^{m}\|_{\widetilde{L}^\infty_{T_{1}}(B^{\sigma+1}_{2,1})}\leq C_{1}A, \label{R-E19}\end{eqnarray}
for all $m\in\mathbb{N}\cup\{0\}$, provided that $\tilde{\kappa}>0$
is sufficiently large, where the constant $C_{1}>0$ independent of
$m$ and
$A:=\|(\rho_{0},\mathbf{u}_{0},\mathbf{E}_{0})\|_{B^{\sigma}_{2,1}}+\|\theta_{0}\|_{B^{\sigma+1}_{2,1}}$.
\end{lem}

\begin{proof}
Indeed, the claim follows from the standard induction. First, we see
that (\ref{R-E19}) holds for $m=0$. Suppose that (\ref{R-E19}) holds
for any $m>0$, we expect to prove it is also true for $m+1$.
Together with the assumption, by (\ref{R-E15})-(\ref{R-E16}) and
(\ref{R-E18}), we get
\begin{eqnarray}
\|(\rho^{m+1},\textbf{u}^{m+1})\|_{\widetilde{L}^{\infty}_{T}(B^{\sigma}_{2,1})}
\leq
C\Big[A+T\max\{2C_{1}A,2(C_{1}A)^2\}\Big]e^{CC_{1}AT},\label{R-E20}
\end{eqnarray}
\begin{eqnarray}
\|\textbf{E}^{m+1}\|_{\widetilde{L}^{\infty}_{T}(B^{\sigma}_{2,1})}\leq
C\Big[A+T\max\{2C_{1}A,2(C_{1}A)^2\}\Big], \label{R-E21}
\end{eqnarray}

\begin{eqnarray}
\|\theta^{m+1}\|_{\widetilde{L}^\infty_{T}(B^{\sigma+1}_{2,1})}&\leq&
C\Big[A+(1+T)\tilde{\kappa}^{-1}\max\{4C_{1}A,4(C_{1}A)^2\}\Big]\nonumber\\&\leq&
C\Big[A+T\max\{4C_{1}A,4(C_{1}A)^2\}\Big], \label{R-E22}
\end{eqnarray}
where we suffice to take $\tilde{\kappa}$ satisfying
$\tilde{\kappa}\geq\frac{(1+T)}{T}$\ ($T$ to be determined) in the
last step of the inequality (\ref{R-E22}). Combining with
(\ref{R-E20})-(\ref{R-E22}), we have
\begin{eqnarray}
&&\|(\rho^{m+1},\textbf{u}^{m+1},\textbf{E}^{m+1})\|_{\widetilde{L}^{\infty}_{T}(B^{\sigma}_{2,1})}
+\|\theta^{m+1}\|_{\widetilde{L}^\infty_{T}(B^{\sigma+1}_{2,1})}\nonumber\\&\leq&
3C\Big[A+T\max\{4C_{1}A,4(C_{1}A)^2\}\Big]e^{CC_{1}AT}.
\label{R-E23}
\end{eqnarray}
Furthermore, if we choose $T_{1}$ satisfying
$$0<T_{1}\leq\min\Big\{\frac{\ln(C_{1}-6C)}{CC_{1}A},T_{0}\Big\}(C_{1}>1+6C),$$
where $T_{0}$ is the root of algebra equation
$$e^{CC_{1}At}=\frac{1}{6C\max\{4,4C_{1}A\}t}\ ,$$
then
$\|(\rho^{m+1},\textbf{u}^{m+1},\textbf{E}^{m+1})\|_{\widetilde{L}^{\infty}_{T_{1}}(B^{\sigma}_{2,1})}
+\|\theta^{m+1}\|_{\widetilde{L}^\infty_{T_{1}}(B^{\sigma+1}_{2,1})}\leq
C_{1}A$ is followed, which concludes the proof of the assertion.
\end{proof}

That is, we find a time $T_{1}>0$ (independent of $m$) such that the
sequence $\{(\rho^{m},\textbf{u}^{m},\theta^{m},\\
\textbf{E}^{m})\}_{m\in \mathbb{N}}$ is uniformly bounded in
$E_{T_{1}}^{\sigma}$.\\

\textbf{Step3: convergence}

Next, it will be shown that
$\{(\rho^{m},\textbf{u}^{m},\theta^{m},\textbf{E}^{m})\}_{m\in
\mathbb{N}}$ is a Cauchy sequence in $E_{T_{1}}^{\sigma}$.

Define
$$\delta\rho^{m+1}=\rho^{m+p+1}-\rho^{m+1},\ \delta\textbf{u}^{m+1}=\textbf{u}^{m+p+1}-\textbf{u}^{m+1},$$$$ \delta\theta^{m+1}=\theta^{m+p+1}-\theta^{m+1},\ \delta\textbf{E}^{m+1}=\textbf{E}^{m+p+1}-\textbf{E}^{m+1},$$
for any $(m,p)\in \mathbb{N}^2$.

Take the difference between the equation (\ref{R-E7}) for the
$(m+p+1)$-th step and the $(m+1)$-th step to give
\begin{equation}
\left\{
\begin{array}{l}\partial_{t}\delta\rho^{m+1}+\textbf{u}^{m+p}\cdot\nabla
\delta\rho^{m+1}+\delta\textbf{u}^{m}\cdot\rho^{m+1}+\mathrm{div}\delta\textbf{u}^{m+1}=0,\\[2mm]
\partial_{t}\delta\textbf{u}^{m+1}+\mathcal{T}_{L}\nabla \delta\rho^{m+1}+(\textbf{u}^{m+p}\cdot\nabla)\delta\textbf{u}^{m+1}+(\delta\textbf{u}^{m}\cdot\nabla)\textbf{u}^{m+1}\\
\hspace{5mm}=-\nabla \delta\theta^{m} -\theta^{m}\delta\nabla
\rho^{m}-\nabla\rho^{m+p}\delta\theta^{m}+\delta\textbf{E}^{m}-\delta\textbf{u}^{m},
\\[2mm]
\partial_{t}\delta\theta^{m+1}-\tilde{\kappa}\Delta\delta\theta^{m+1}\\ \hspace{5mm}=-\delta\textbf{u}^{m}\cdot\nabla
\theta^{m+p}-\textbf{u}^{m}\nabla\delta\theta^{m}+[h_{1}(\rho^{m+p})-h_{1}(\rho^{m})]\Delta\theta^{m+p}\\
\hspace{10mm}+h_{1}(\rho^m)\Delta\delta\theta^m-(\gamma-1)(\mathcal{T}_{L}+\theta^{m})
\mathrm{div}\delta\textbf{u}^{m}\\
\hspace{10mm}-\delta\theta^{m}\mathrm{div}\textbf{u}^{m+p}
+\frac{\gamma-1}{2}(\textbf{u}^{m+p}+\textbf{u}^{m})\delta\textbf{u}^{m}-\delta\theta^{m},\\[2mm]
\partial_{t}\delta\textbf{E}^{m+1}=-\nabla\Delta^{-1}\mathrm{div}\{[h_{2}(\rho^{m+p})-h_{2}(\rho^{m})]\textbf{u}^{m+p}+h_{2}(\rho^{m})\delta\textbf{u}^{m}+\bar{n}\delta\textbf{u}^{m}\},
 \end{array} \right.\label{R-E24}
\end{equation}
subject to the initial data
\begin{equation}
(\delta\rho^{m+1},\delta\textbf{u}^{m+1},\delta\theta^{m+1},\delta\textbf{E}^{m+1})(x,0)=[S_{m+p+1}-S_{m+1}](\rho_{0},\textbf{u}_{0},\theta_{0},\textbf{E}_{0}),\
x\in \mathbb{R}^{N}. \label{R-E25}
\end{equation}
Applying the operator $\Delta_{q}(q\geq-1)$ to the first two
equations of (\ref{R-E24}) gives
\begin{equation}
\left\{
\begin{array}{l}\partial_{t}\Delta_{q}\delta\rho^{m+1}+(\textbf{u}^{m+p}\cdot\nabla)\Delta_{q}\delta\rho^{m+1}+\Delta_{q}\mathrm{div}\delta\textbf{u}^{m+1}\\
=[\textbf{u}^{m+p},\Delta_{q}]\cdot\nabla\delta\rho^{m+1}-\Delta_{q}(\delta\textbf{u}^{m}\cdot\nabla\rho^{m+1}),\\[2mm]
\partial_{t}\Delta_{q}\delta\textbf{u}^{m+1}+\mathcal{T}_{L}\Delta_{q}\nabla \delta\rho^{m+1}+(\textbf{u}^{m+p}\cdot\nabla)\Delta_{q}\delta\textbf{u}^{m+1}\\
=[\textbf{u}^{m+p},\Delta_{q}]\cdot\nabla\delta\textbf{u}^{m+1}-\Delta_{q}(\delta\textbf{u}^{m}\cdot\nabla\textbf{u}^{m+1})-\Delta_{q}\nabla
\delta\theta^{m}\\
\hspace{5mm}-\Delta_{q}(\delta\theta^m\nabla\rho^{m+p})-\nabla\delta\rho^{m}\Delta_{q}\theta^m+
[\nabla\delta\rho^{m},\Delta_{q}]\theta^{m}+\Delta_{q}\delta\textbf{E}^{m}-\Delta_{q}\delta\textbf{u}^{m},
\end{array} \right.\label{R-E26}
\end{equation}
By multiplying the first equation of Eqs. (\ref{R-E26}) by
$\mathcal{T}_{L}\Delta_{q}\delta\rho^{m+1}$, the second one by
$\Delta_{q}\delta\textbf{u}^{m+1}$, and adding the resulting
equations together, after integrating it over $\mathbb{R}^{N}$, we
have
\begin{eqnarray}
&&\frac{1}{2}\frac{d}{dt}\Big(\mathcal{T}_{L}\|\Delta_{q}\delta\rho^{m+1}\|^2_{L^2}+\|\Delta_{q}\delta\textbf{u}^{m+1}\|^2_{L^2}\Big)\nonumber\\&\leq&
\|\nabla\textbf{u}^{m+p}\|_{L^{\infty}}\Big(\mathcal{T}_{L}\|\Delta_{q}\delta\rho^{m+1}\|^2_{L^2}+\|\Delta_{q}\delta\textbf{u}^{m+1}\|^2_{L^2}\Big)\nonumber\\&&+
\Big(\mathcal{T}_{L}\|[\textbf{u}^{m+p},\Delta_{q}]\cdot\nabla\delta\rho^{m+1}\|_{L^2}+\|\Delta_{q}(\delta\textbf{u}^{m}\cdot\nabla\rho^{m+1})\|_{L^2}\Big)\|\Delta_{q}\delta\rho^{m+1}\|_{L^2}
\nonumber\\&&+\Big\{\|[\textbf{u}^{m+p},\Delta_{q}]\cdot\nabla\delta\textbf{u}^{m+1}\|_{L^2}+\|\Delta_{q}(\delta\textbf{u}^{m}\cdot\nabla\textbf{u}^{m+1})\|_{L^2}
+\|\Delta_{q}\nabla
\delta\theta^{m}\|_{L^2}\nonumber\\&&+\|\Delta_{q}(\delta\theta^m\nabla\rho^{m+p})\|_{L^2}+\|\delta\rho^{m}\|_{L^\infty}\Big(\|\Delta_{q}\nabla\theta^m\|_{L^2}+2^{q}\|\Delta_{q}\theta^m\|_{L^2}\Big)\nonumber\\&&+
\|[\nabla
\delta\rho^{m},\Delta_{q}]\theta^{m}\|_{L^2}+\|\Delta_{q}\delta\textbf{E}^{m}\|_{L^2}+\|\Delta_{q}\delta\textbf{u}^{m}\|_{L^2}\Big\}\|\Delta_{q}\delta\textbf{u}^{m+1}\|_{L^2},
\label{R-E27}
\end{eqnarray}where we have bounded the integration
\begin{eqnarray*}&&-\int\Delta_{q}\theta^m\nabla\delta\rho^{m}\cdot\Delta_{q}\delta\textbf{u}^{m+1}\nonumber\\
&=&\int\delta\rho^{m}\Delta_{q}\nabla\theta^m\cdot\Delta_{q}\delta\textbf{u}^{m+1}
+\delta\rho^{m}\Delta_{q}\theta^m\cdot\Delta_{q}\mathrm{div}\delta\textbf{u}^{m+1}\nonumber\\&\leq&\|\delta\rho^{m}\|_{L^\infty}\Big(\|\Delta_{q}\nabla\theta^m\|_{L^2}+2^{q}\|\Delta_{q}\theta^m\|_{L^2}\Big)\|\Delta_{q}\delta\textbf{u}^{m+1}\|_{L^2}.\end{eqnarray*}
Similar to the estimate of (\ref{R-E11}), we can obtain
\begin{eqnarray}
&&\frac{d}{dt}\Big(\mathcal{T}_{L}\|\Delta_{q}\delta\rho^{m+1}\|^2_{L^2}+\|\Delta_{q}\delta\textbf{u}^{m+1}\|^2_{L^2}+\varepsilon\Big)^{\frac{1}{2}}\nonumber\\&\leq&
\|\nabla\textbf{u}^{m+p}\|_{L^{\infty}}\Big(\mathcal{T}_{L}\|\Delta_{q}\delta\rho^{m+1}\|_{L^2}+\|\Delta_{q}\delta\textbf{u}^{m+1}\|_{L^2}\Big)+
\Big(\mathcal{T}_{L}\|[\textbf{u}^{m+p},\Delta_{q}]\cdot\nabla\delta\rho^{m+1}\|_{L^2}\nonumber\\&&+\|\Delta_{q}(\delta\textbf{u}^{m}\cdot\nabla\rho^{m+1})\|_{L^2}\Big)
+\Big\{\|[\textbf{u}^{m+p},\Delta_{q}]\cdot\nabla\delta\textbf{u}^{m+1}\|_{L^2}+\|\Delta_{q}(\delta\textbf{u}^{m}\cdot\nabla\textbf{u}^{m+1})\|_{L^2}
\nonumber\\&&+\|\Delta_{q}\nabla
\delta\theta^{m}\|_{L^2}+\|\Delta_{q}(\delta\theta^m\nabla\rho^{m+p})\|_{L^2}+\|\delta\rho^{m}\|_{L^\infty}\Big(\|\Delta_{q}\nabla\theta^m\|_{L^2}+2^{q}\|\Delta_{q}\theta^m\|_{L^2}\Big)\nonumber\\&&+
\|[\nabla
\delta\rho^{m},\Delta_{q}]\theta^{m}\|_{L^2}+\|\Delta_{q}\delta\textbf{E}^{m}\|_{L^2}+\|\Delta_{q}\delta\textbf{u}^{m}\|_{L^2}\Big\},
\label{R-E28}\end{eqnarray} where $\varepsilon>0$ is a small
quantity.

Integrating (\ref{R-E28}) with respect to the variable
$t\in[0,T_{1}]$, then taking $\varepsilon\rightarrow0$, we arrive at
\begin{eqnarray}
&&\|\Delta_{q}\delta\rho^{m+1}(t)\|_{L^2}+\|\Delta_{q}\delta\textbf{u}^{m+1}(t)\|_{L^2}\nonumber\\&\leq&C\Big(\|\Delta_{q}\delta\rho^{m+1}_{0}\|_{L^2}+\|\Delta_{q}\delta\textbf{u}^{m+1}_{0}\|_{L^2}\Big)
+C\int^t_{0}\|\nabla\textbf{u}^{m+p}\|_{L^{\infty}}\Big(\|\Delta_{q}\delta\rho^{m+1}\|_{L^2}\nonumber\\&&+\|\Delta_{q}\delta\textbf{u}^{m+1}\|_{L^2}\Big)+
C\int^t_{0}\Big(\|[\textbf{u}^{m+p},\Delta_{q}]\cdot\nabla\delta\rho^{m+1}\|_{L^2}+\|\Delta_{q}(\delta\textbf{u}^{m}\cdot\nabla\rho^{m+1})\|_{L^2}\Big)
\nonumber\\&&+C\int^t_{0}\Big(\|[\textbf{u}^{m+p},\Delta_{q}]\cdot\nabla\delta\textbf{u}^{m+1}\|_{L^2}+\|\Delta_{q}(\delta\textbf{u}^{m}\cdot\nabla\textbf{u}^{m+1})\|_{L^2}
+\|\Delta_{q}\nabla
\delta\theta^{m}\|_{L^2}\nonumber\\&&+\|\Delta_{q}(\delta\theta^m\nabla\rho^{m+p})\|_{L^2}+\|\delta\rho^{m}\|_{L^\infty}(\|\Delta_{q}\nabla\theta^m\|_{L^2}+2^{q}\|\Delta_{q}\theta^m\|_{L^2})\nonumber\\&&+
\|[\nabla
\delta\rho^{m},\Delta_{q}]\theta^{m}\|_{L^2}+\|\Delta_{q}\delta\textbf{E}^{m}\|_{L^2}+\|\Delta_{q}\delta\textbf{u}^{m}\|_{L^2}\Big).
\label{R-E29}
\end{eqnarray}
By multiplying the factor $2^{q(\sigma-1)}$ on both sides of the
resulting inequality (\ref{R-E29}), we obtain
\begin{eqnarray}
&&2^{q(\sigma-1)}\Big(\|\Delta_{q}\delta\rho^{m+1}(t)\|_{L^2}+\|\Delta_{q}\delta\textbf{u}^{m+1}(t)\|_{L^2}\Big)\nonumber\\&\leq&C2^{q(\sigma-1)}\Big(\|\Delta_{q}\delta\rho^{m+1}_{0}\|_{L^2}+\|\Delta_{q}\delta\textbf{u}^{m+1}_{0}\|_{L^2}\Big)
\nonumber\\&&+C\int^t_{0}\|\nabla\textbf{u}^{m+p}\|_{L^{\infty}}2^{q(\sigma-1)}\Big(\|\Delta_{q}\delta\rho^{m+1}\|_{L^2}+\|\Delta_{q}\delta\textbf{u}^{m+1}\|_{L^2}\Big)\nonumber\\&&+
C\int^t_{0}\Big(c_{q}\|\textbf{u}^{m+p}\|_{B^{\sigma}_{2,1}}\|\delta\rho^{m+1}\|_{_{B^{\sigma-1}_{2,1}}}+c_{q}\|\delta\textbf{u}^{m}\|_{B^{\sigma-1}_{2,1}}\|\rho^{m+1}\|_{_{B^{\sigma}_{2,1}}}\Big)
\nonumber\\&&+C\int^t_{0}\Big(c_{q}\|\textbf{u}^{m+p}\|_{B^{\sigma}_{2,1}}\|\delta\textbf{u}^{m+1}\|_{B^{\sigma-1}_{2,1}}+c_{q}\|\delta\textbf{u}^{m}\|_{B^{\sigma-1}_{2,1}}\|\textbf{u}^{m+1}\|_{B^{\sigma}_{2,1}}
\nonumber\\&&+c_{q}\|
\delta\theta^{m}\|_{B^{\sigma}_{2,1}}+c_{q}\|\delta\theta^m\|_{B^{\sigma}_{2,1}}\|\rho^{m+p}\|_{B^{\sigma}_{2,1}}+c_{q}\|\delta\rho^{m}\|_{L^\infty}(\|\nabla\theta^m\|_{B^{\sigma-1}_{2,1}}+\|\theta^m\|_{B^{\sigma}_{2,1}})\nonumber\\&&+
c_{q}\|\nabla\theta^{m}\|_{B^{\sigma-1}_{2,1}}
\|\delta\rho^{m}\|_{B^{\sigma-1}_{2,1}}+2^{q(\sigma-1)}\|\Delta_{q}\delta\textbf{E}^{m}\|_{L^2}+2^{q(\sigma-1)}\|\Delta_{q}\delta\textbf{u}^{m}\|_{L^2}\Big),\label{R-E30}
\end{eqnarray}
where $\{c_{q}\}$ denotes some sequence which satisfies
$\|(c_{q})\|_{ {l^{1}}}\leq 1$.

Summing up (\ref{R-E30}) on $q\geq-1$, it is not difficult to get
\begin{eqnarray}
&&\|(\delta\rho^{m+1},\delta\textbf{u}^{m+1})\|_{\widetilde{L}^\infty_{T_{1}}(B^{\sigma-1}_{2,1})}
\nonumber\\&\leq&C\|(\delta\rho^{m+1}_{0},\delta\textbf{u}^{m+1}_{0})\|_{B^{\sigma-1}_{2,1}}
\nonumber\\&&+C\int^{T_{1}}_{0}\|\textbf{u}^{m+p}\|_{B^{\sigma}_{2,1}}\|(\delta\rho^{m+1},\delta\textbf{u}^{m+1})\|_{\widetilde{L}^\infty_{t}(B^{\sigma-1}_{2,1})}dt
\nonumber\\&&+
C\int^{T_{1}}_{0}\|\delta\textbf{u}^{m}\|_{B^{\sigma-1}_{2,1}}\Big(1+\|\rho^{m+1}\|_{B^{\sigma}_{2,1}}+\|\textbf{u}^{m+1}\|_{B^{\sigma}_{2,1}}\Big)dt
\nonumber\\&&+C\int^{T_{1}}_{0}\|\delta\theta^m\|_{B^{\sigma}_{2,1}}(1+\|\rho^{m+p}\|_{B^{\sigma}_{2,1}})dt
\nonumber\\&&+C\int^{T_{1}}_{0}\Big(\|\delta\rho^{m}\|_{B^{\sigma-1}_{2,1}}\|\theta^m\|_{B^{\sigma}_{2,1}}
+\|\delta\textbf{E}^{m}\|_{B^{\sigma-1}_{2,1}}\Big)dt\nonumber\\&\leq&
C2^{-m}\|(\rho_{0},\textbf{u}_{0})\|_{B^{\sigma}_{2,1}}\nonumber\\&&
+C\int^{T_{1}}_{0}\|\textbf{u}^{m+p}\|_{B^{\sigma}_{2,1}}\|(\delta\rho^{m+1},\delta\textbf{u}^{m+1})\|_{\widetilde{L}^\infty_{t}(B^{\sigma-1}_{2,1})}dt
\nonumber\\&&+C\int^{T_{1}}_{0}\Big(\|(\delta\rho^{m},\delta\textbf{u}^{m},\delta\textbf{E}^{m})\|_{B^{\sigma-1}_{2,1}}
+\|\delta\theta^m\|_{B^{\sigma}_{2,1}}\Big)
\nonumber\\&&\hspace{20mm}\times\Big(1+\|(\rho^{m+1},\rho^{m+p},\textbf{u}^{m+1},\theta^m)\|_{B^{\sigma}_{2,1}}\Big)dt,\label{R-E31}
\end{eqnarray}
where we have used Lemma \ref{lem2.1} and Remark \ref{rem2.1}.

With the aid of Gronwall's inequality, we immediately  deduce that
\begin{eqnarray}
&&\|(\delta\rho^{m+1},\delta\textbf{u}^{m+1})\|_{\widetilde{L}^\infty_{T_{1}}(B^{\sigma-1}_{2,1})}\nonumber\\&\leq&Ce^{CZ^{m+p}(T_{1})}
\Big\{2^{-m}\|(\rho_{0},\textbf{u}_{0})\|_{B^{\sigma}_{2,1}}\nonumber\\&&
+\int^{T_{1}}_{0}e^{-CZ^{m+p}(t)}\Big(\|(\delta\rho^{m},\delta\textbf{u}^{m},\delta\textbf{E}^{m})\|_{B^{\sigma-1}_{2,1}}
+\|\delta\theta^m\|_{B^{\sigma}_{2,1}}\Big)
\nonumber\\&&\hspace{20mm}\times\Big(1+\|(\rho^{m+1},\rho^{m+p},\textbf{u}^{m+1},\theta^m)\|_{B^{\sigma}_{2,1}}\Big)dt\Big\}
\nonumber\\&\leq&Ce^{CT_{1}}\Big\{2^{-m}
+T_{1}\Big(\|(\delta\rho^{m},\delta\textbf{u}^{m},\delta\textbf{E}^{m})\|_{\widetilde{L}^\infty_{T_{1}}(B^{\sigma-1}_{2,1})}
+\|\delta\theta^m\|_{\widetilde{L}^\infty_{T_{1}}(B^{\sigma}_{2,1})}\Big)\Big\},\label{R-E32}
\end{eqnarray}
where we have noticed Remark \ref{rem2.2} and the fact that the
sequence
$\{(\rho^{m},\textbf{u}^{m},\theta^{m},\textbf{E}^{m})\}_{m\in
\mathbb{N}}$ is uniformly bounded in $E_{T_{1}}^{\sigma}$.

From the last equation of (\ref{R-E24}), we get directly
\begin{eqnarray}
&&\|\delta\textbf{E}^{m+1}\|_{\widetilde{L}^\infty_{T_{1}}(B^{\sigma-1}_{2,1})}\nonumber\\&\leq&\|\delta\textbf{E}^{m+1}_{0}\|_{B^{\sigma-1}_{2,1}}
+\int^{T_{1}}_{0}\Big(\|\delta\rho^m\|_{B^{\sigma-1}_{2,1}}\|\textbf{u}^{m+p}\|_{B^{\sigma}_{2,1}}
+\|\delta\textbf{u}^m\|_{B^{\sigma-1}_{2,1}}(1+\|\rho^{m}\|_{B^{\sigma}_{2,1}})\Big)dt
\nonumber\\&\leq&C2^{-m}+CT_{1}\|(\delta\rho^m,\delta\textbf{u}^m)\|_{\widetilde{L}^\infty_{T_{1}}(B^{\sigma-1}_{2,1})}.\label{R-E33}
\end{eqnarray}
Using Proposition \ref{prop2.6} (taking $\alpha_{1}=\alpha=\infty,\
s=\sigma, \ p=2$ and $r=1$), we have
\begin{eqnarray}
&&\|\delta\theta^{m+1}\|_{\widetilde{L}^\infty_{T_{1}}(B^{\sigma}_{2,1})}\nonumber\\&\leq&
C\Big(\|\delta\theta^{m+1}_{0}\|_{B^{\sigma}_{2,1}}
+(1+T_{1})\tilde{\kappa}^{-1}\|F_{2}^{m}\|_{\widetilde{L}^\infty_{T_{1}}(B^{\sigma-2}_{2,1})}\Big),\label{R-E34}
\end{eqnarray}
where \begin{eqnarray*}F^{m}_{2}&:=&-\delta\textbf{u}^{m}\cdot\nabla
\theta^{m+p}-\textbf{u}^{m}\nabla\delta\theta^{m}+[h_{1}(\rho^{m+p})-h_{1}(\rho^{m})]\Delta\theta^{m+p}
\nonumber\\&&+h_{1}(\rho^m)\Delta\delta\theta^m-(\gamma-1)(\mathcal{T}_{L}+\theta^{m})
\mathrm{div}\delta\textbf{u}^{m}-\delta\theta^{m}\mathrm{div}\textbf{u}^{m+p}
\nonumber\\&&+\frac{\gamma-1}{2}(\textbf{u}^{m+p}+\textbf{u}^{m})\delta\textbf{u}^{m}-\delta\theta^{m}.\end{eqnarray*}

In bounding $F^{m}_{2}$, each product term can be estimated
effectively with the help of the standard Moser-type inequality
(Proposition \ref{prop2.3}), except for the term
$h_{1}(\rho^m)\Delta\delta\theta^m$.
 Here,
we develop a Moser-type inequality of general form to estimate
$h_{1}(\rho^m)\Delta\delta\theta^m$, which will be shown in the
Appendix, see Proposition \ref{prop4.1}. According to it, we can
reach
\begin{eqnarray}
\|h_{1}(\rho^m)\Delta\delta\theta^m\|_{B^{\sigma-2}_{2,1}}\leq C
(\|h_{1}(\rho^m)\|_{L^\infty}\|\Delta\delta\theta^m\|_{B^{\sigma-2}_{2,1}}
+\|\Delta\delta\theta^m\|_{L^2}\|h_{1}(\rho^m)\|_{B^{\sigma-2}_{\infty,1}}).\label{R-E35}
\end{eqnarray}
The second term in (\ref{R-E35}) can be further estimated as
\begin{eqnarray}
&&\|\Delta\delta\theta^m\|_{L^2}\|h_{1}(\rho^m)\|_{B^{\sigma-2}_{\infty,1}}\nonumber\\&\leq&
C\|\Delta\delta\theta^m\|_{B^{\sigma-2}_{2,1}}\|h_{1}(\rho^m)\|_{B^{N-1}_{2,1}}\
(\sigma=1+N/2,\ N\geq2)\nonumber\\&\leq&
C\|\delta\theta^m\|_{B^{\sigma}_{2,1}}\|h_{1}(\rho^m)\|_{B^{\sigma}_{2,1}}\nonumber\\&\leq&
C\|\delta\theta^m\|_{B^{\sigma}_{2,1}}\|\rho^m\|_{B^{\sigma}_{2,1}},\label{R-E36}
\end{eqnarray}
where we used the embedding properties
$B^{\sigma-2}_{2,1}\hookrightarrow L^2$ and
$B^{N-1}_{2,1}\hookrightarrow B^{\sigma-2}_{\infty,1}$. To ensure
$B^{\sigma}_{2,1}\hookrightarrow B^{N-1}_{2,1}$, $N-1\leq 1+N/2$
i.e. $N\leq4$ is required in the last second step of (\ref{R-E36}).

Combining (\ref{R-E34})-(\ref{R-E36}) and recalling on the choice of
$\tilde{\kappa}(\tilde{\kappa}>\frac{1+T_{1}}{T_{1}})$, we conclude
that
\begin{eqnarray}
&&\|\delta\theta^{m+1}\|_{\widetilde{L}^\infty_{T_{1}}(B^{\sigma}_{2,1})}\nonumber\\&\leq&
C\Big\{2^{-m}+T_{1}\Big(\|(\delta\rho^{m},\delta\textbf{u}^{m})\|_{\widetilde{L}^\infty_{T_{1}}(B^{\sigma-1}_{2,1})}
+\|\delta\theta^{m}\|_{\widetilde{L}^\infty_{T_{1}}(B^{\sigma}_{2,1})}\Big)\Big\},\label{R-E37}
\end{eqnarray}
Therefore, together with (\ref{R-E32})-(\ref{R-E33}) and
(\ref{R-E37}), we end up with
\begin{eqnarray}
&&\|(\delta\rho^{m+1},\delta\textbf{u}^{m+1},\delta\textbf{E}^{m+1})\|_{\widetilde{L}^\infty_{T_{1}}(B^{\sigma-1}_{2,1})}+
\|\delta\theta^{m+1}\|_{\widetilde{L}^\infty_{T_{1}}(B^{\sigma}_{2,1})}\nonumber\\&\leq&
C_{T_{1}}\Big\{2^{-m}+T_{1}\Big(\|(\delta\rho^{m},\delta\textbf{u}^{m},\delta\textbf{E}^{m})\|_{\widetilde{L}^\infty_{T_{1}}(B^{\sigma-1}_{2,1})}
+\|\delta\theta^m\|_{\widetilde{L}^\infty_{T_{1}}(B^{\sigma}_{2,1})}\Big)\Big\},
\label{R-E38}
\end{eqnarray} where $C_{T_{1}}:=Ce^{CT_{1}}$.
Arguing by induction, one can easily deduce that
\begin{eqnarray}
&&\|(\delta\rho^{m+1},\delta\textbf{u}^{m+1},\delta\textbf{E}^{m+1})\|_{\widetilde{L}^\infty_{T_{1}}(B^{\sigma-1}_{2,1})}+
\|\delta\theta^{m+1}\|_{\widetilde{L}^\infty_{T_{1}}(B^{\sigma}_{2,1})}\nonumber\\&\leq&
\frac{(T_{1}C_{T_{1}})^{m+1}}{(m+1)!}\Big(\|(\delta\rho^{p},\delta\textbf{u}^{p},\delta\textbf{E}^{p})\|_{\widetilde{L}^\infty_{T_{1}}(B^{\sigma-1}_{2,1})}
+\|\delta\theta^p\|_{\widetilde{L}^\infty_{T_{1}}(B^{\sigma}_{2,1})}\Big)
\nonumber\\&&+C_{T_{1}}\sum^{m}_{k=0}2^{-(m-k)}\frac{(T_{1}C_{T_{1}})^{k}}{k!}.\label{R-E39}
\end{eqnarray}
As
$\|(\delta\rho^{p},\delta\textbf{u}^{p},\delta\textbf{E}^{p})\|_{\widetilde{L}^\infty_{T_{1}}(B^{\sigma-1}_{2,1})}
+\|\delta\theta^p\|_{\widetilde{L}^\infty_{T_{1}}(B^{\sigma}_{2,1})}$
can be bounded independent of $p$, we further take $T_{1}$ so small
that $$\frac{(T_{1}C_{T_{1}})^{m+1}}{(m+1)!}\leq C2^{-m}\ \
\mbox{and}\ \  \frac{C_{T_{1}}(T_{1}C_{T_{1}})^{k}}{k!}\leq
4^{-k}.$$ Thus we conclude that there exists some constant $C_{2}>0$
(independent of $m$) such that
\begin{eqnarray}
&&\|(\delta\rho^{m+1},\delta\textbf{u}^{m+1},\delta\textbf{E}^{m+1})\|_{\widetilde{L}^\infty_{T_{1}}(B^{\sigma-1}_{2,1})}+
\|\delta\theta^{m+1}\|_{\widetilde{L}^\infty_{T_{1}}(B^{\sigma}_{2,1})}\leq
C_{2}2^{-m},\label{R-E40}
\end{eqnarray}
which implies
$\{(\rho^{m},\textbf{u}^{m},\theta^{m},\textbf{E}^{m})\}_{m\in
\mathbb{N}}$ is a Cauchy sequence in $E^{\sigma-1}_{T_{1}}$.
Therefore, there exists some function
$(\rho,\textbf{u},\theta,\textbf{E})$ in $E^{\sigma-1}_{T_{1}}$ such
that
$$\{(\rho^{m},\textbf{u}^{m},\theta^{m},\textbf{E}^{m})\}\rightarrow(\rho,\textbf{u},\theta,\textbf{E})\ \  \mbox{strongly in}\ \ E^{\sigma-1}_{T_{1}}.$$

\noindent\textbf{Step4: the solution
$(\rho,\textbf{u},\theta,\textbf{E})$}

In this step we show that $(\rho,\textbf{u},\theta,\textbf{E})\in
E^{\sigma}_{T_{1}}$ is a solution of the system
(\ref{R-E5})-(\ref{R-E6}). Fatou's property ensures that
$(\rho,\textbf{u},\theta,\textbf{E})\in
\widetilde{L}^\infty_{T_{1}}(B^{\sigma}_{2,1})\times(\widetilde{L}^\infty_{T_{1}}(B^{\sigma}_{2,1}))^N\times
\widetilde{L}^\infty_{T_{1}}(B^{\sigma+1}_{2,1})\times(\widetilde{L}^\infty_{T_{1}}(B^{\sigma}_{2,1}))^N$,
since $\{(\rho^{m},\textbf{u}^{m},\theta^{m},
\textbf{E}^{m})\}_{m\in \mathbb{N}}$ is also uniformly bounded in
the spaces
$\widetilde{L}^\infty_{T_{1}}(B^{\sigma}_{2,1})\times(\widetilde{L}^\infty_{T_{1}}(B^{\sigma}_{2,1}))^N
\times\widetilde{L}^\infty_{T_{1}}(B^{\sigma+1}_{2,1})\times(\widetilde{L}^\infty_{T_{1}}(B^{\sigma}_{2,1}))^N$.

On the other hand,
$\{(\rho^{m},\textbf{u}^{m},\textbf{E}^{m})\}_{m\in \mathbb{N}}$
converges to $(\rho,\textbf{u},\textbf{E})$ in
$\mathcal{C}([0,T_{1}];B^{\sigma-1}_{2,1})$ and
$\{\theta^{m}\}_{m\in \mathbb{N}}$ converges to $\theta$ in
$\mathcal{C}([0,T_{1}];B^{\sigma}_{2,1})$. These properties of
strong convergence enable us to pass to the limits in the system
(\ref{R-E5})-(\ref{R-E6}) and conclude that
$(\rho,\textbf{u},\theta,\textbf{E})$ to the system
(\ref{R-E5})-(\ref{R-E6}). Now, what remains is to check
$(\rho,\textbf{u},\theta,\textbf{E})$ also belongs to
$\mathcal{C}([0,T_{1}];B^{\sigma}_{2,1})\times(\mathcal{C}([0,T_{1}];B^{\sigma}_{2,1}))^N\times
\mathcal{C}([0,T_{1}];B^{\sigma+1}_{2,1})\times(\mathcal{C}([0,T_{1}];B^{\sigma}_{2,1}))^N$.
Indeed, for instance, we easily achieve that the map
$t\mapsto\|\Delta_{q}\rho(t)\|_{L^2}$ is continuous on $[0,T_{1}]$,
since $\rho\in \mathcal{C}([0,T_{1}];B^{\sigma-1}_{2,1})$. Then we
have $\Delta_{q}\rho(t)\in \mathcal{C}([0,T_{1}];B^{\sigma}_{2,1})$
for all $q\geq-1$. Note that
$\rho\in\widetilde{L}^\infty_{T_{1}}(B^{\sigma}_{2,1})$, the series
$\sum_{q\geq-1}2^{q\sigma}\|\Delta_{q}\rho(t)\|_{L^2}$ converges
uniformly on $[0,T_{1}]$, which yields $\rho\in
\mathcal{C}([0,T_{1}];B^{\sigma}_{2,1})$. The same arguments are
valid for the other variables $(\textbf{u},\theta,\textbf{E})$.
Hence, we finish the existence part of
solutions. \\

\noindent\textbf{Step5: uniqueness}

 Let $\widetilde{\rho}=\rho_{1}-\rho_{2},\
\widetilde{\textbf{u}}=\textbf{u}_{1}-\textbf{u}_{2},\
\widetilde{\theta}=\theta_{1}-\theta_{2}, \
\widetilde{\textbf{E}}=\textbf{E}_{1}-\textbf{E}_{2}$ where
$(\rho_{1},\textbf{u}_{1},\theta_{1},\textbf{E}_{1})^{\top}$ and
$(\rho_{2},\textbf{u}_{2},\theta_{2},\textbf{E}_{2})^{\top}$ are two
solutions to the system (\ref{R-E5})-(\ref{R-E6}) subject to the
same initial data, respectively. Then the error solution
$(\widetilde{\rho}, \widetilde{\textbf{u}},\widetilde{\theta},
\widetilde{\textbf{E}})^{\top}$ satisfies
\begin{equation}
\left\{
\begin{array}{l}
\partial_{t}\widetilde{\rho}+\mbox{div}\widetilde{\textbf{u}}=-\textbf{u}_{1}\cdot\nabla
\widetilde{\rho}-\widetilde{\textbf{u}}\cdot\nabla\rho_{2},\\[2mm]
\partial_{t}\widetilde{\textbf{u}}+\mathcal{T}_{L}\nabla
\widetilde{\rho}=-\nabla\widetilde{\theta}-\textbf{u}_{1}\cdot\nabla\widetilde{\textbf{u}}-\widetilde{\textbf{u}}\nabla\textbf{u}_{2}-\theta_{1}\nabla
\widetilde{\rho}-\widetilde{\theta}\nabla
\rho_{2}+\widetilde{\textbf{E}}-\widetilde{\textbf{u}},
\\[2mm]
\partial_{t}\widetilde{\theta}-\widetilde{\kappa}\Delta\widetilde{\theta}
=-\textbf{u}_{1}\cdot\nabla\widetilde{\theta}-\widetilde{\textbf{u}}\nabla\theta_{2}+[h_{1}(\rho_{1})-h_{1}(\rho_{2})]\Delta\theta_{1}
+h_{1}(\rho_{2})\Delta\widetilde{\theta}\\
\hspace{20mm}-(\gamma-1)(\mathcal{T}_{L}+\theta_{1})\mbox{div}\widetilde{\textbf{u}}
-(\gamma-1)\widetilde{\theta}\mbox{div}\textbf{u}_{2}
+\frac{(\gamma-1)}{2}\widetilde{\textbf{u}}(\textbf{u}_{1}+\textbf{u}_{2})-\widetilde{\theta},
\\ [2mm] \partial_{t}\widetilde{\textbf{E}}=-\nabla\Delta^{-1}\nabla\cdot[(h_{2}(\rho_1)-h_{2}(\rho_2))\textbf{u}_{1}
+(h_{2}(\rho_{2})+\bar{n})\widetilde{\textbf{u}}].
\end{array}
\right.\label{R-E41}
\end{equation}
As previously, following from the proof of Cauchy sequence, we
obtain the inequalities:
\begin{eqnarray}\|(\widetilde{\rho},\widetilde{\textbf{u}})\|_{\widetilde{L}^\infty_{T_{1}}(B^{\sigma-1}_{2,1})}&\leq& C
\int^{T_{1}}_{0}\Big(\|(\widetilde{\rho},\widetilde{\textbf{u}},\widetilde{\textbf{E}})\|_{B^{\sigma-1}_{2,1}}\Big)
\Big(1+\|(\rho_{1},\rho_{2},\textbf{u}_{1},\textbf{u}_{2},\theta_{1})\|_{B^{\sigma}_{2,1}}\Big)dt
\nonumber\\&&+\|\widetilde{\theta}\|_{L^1_{T_{1}}(B^{\sigma}_{2,1})}\label{R-E42}
\end{eqnarray}
and
\begin{equation}\|\widetilde{\textbf{E}}\|_{\widetilde{L}^\infty_{T_{1}}(B^{\sigma-1}_{2,1})}\leq C
\int^{T_{1}}_{0}\|(\widetilde{\rho},\widetilde{\textbf{u}})\|_{B^{\sigma-1}_{2,1}}
\Big(1+\|(\rho_{2},\textbf{u}_{1})\|_{B^{\sigma}_{2,1}}\Big)dt.\label{R-E43}
\end{equation}
According to Proposition \ref{prop2.6} (taking
$\alpha_{1}=\alpha=1,\ s=\sigma-2, \ p=2$ and $r=1$), we have
\begin{eqnarray}
\|\widetilde{\theta}\|_{\widetilde{L}^1_{T_{1}}(B^{\sigma}_{2,1})}&\leq&
CT_{1}\Big\{(1+\|(\rho_{2},\textbf{u}_{1},\textbf{u}_{2})\|_{\widetilde{L}^\infty_{T_{1}}(B^{\sigma-1}_{2,1})})
\|\widetilde{\theta}\|_{\widetilde{L}^1_{T_{1}}(B^{\sigma}_{2,1})}
\nonumber\\&&+\int^{T_{1}}_{0}\|(\widetilde{\rho},\widetilde{\textbf{u}})\|_{B^{\sigma-1}_{2,1}}
(1+\|(\textbf{u}_{1},\textbf{u}_{2})\|_{B^{\sigma}_{2,1}}+\|(\theta_{1},\theta_{2})\|_{B^{\sigma}_{2,1}})dt\Big\}.\label{R-E44}
\end{eqnarray}
We further choose $T_{1}$ satisfying
$$T_{1}\leq\frac{1}{2C(1+\|(\rho_{2},\textbf{u}_{1},\textbf{u}_{2})\|_{\widetilde{L}^\infty_{T_{1}}(B^{\sigma-1}_{2,1})})},$$
and obtain
\begin{eqnarray}
\|\widetilde{\theta}\|_{\widetilde{L}^1_{T_{1}}(B^{\sigma}_{2,1})}
&\leq&C\int^{T_{1}}_{0}\|(\widetilde{\rho},\widetilde{\textbf{u}})\|_{B^{\sigma-1}_{2,1}}(1+\|(\textbf{u}_{1},\textbf{u}_{2})\|_{B^{\sigma}_{2,1}}
+\|(\theta_{1},\theta_{2})\|_{B^{\sigma}_{2,1}})dt. \label{R-E45}
\end{eqnarray}
Note that $\widetilde{L}^1_{T_{1}}(B^{\sigma}_{2,1})\equiv
L^1_{T_{1}}(B^{\sigma}_{2,1})$, we insert (\ref{R-E45}) into
(\ref{R-E42}) to get after combining (\ref{R-E43})
\begin{eqnarray}
&&\|(\widetilde{\rho},\widetilde{\textbf{u}},\widetilde{\textbf{E}})\|_{\widetilde{L}^\infty_{T_{1}}(B^{\sigma-1}_{2,1})}\nonumber\\&\leq&
C\int^{T_{1}}_{0}\|(\widetilde{\rho},\widetilde{\textbf{u}},\widetilde{\textbf{E}})\|_{\widetilde{L}^\infty_{t}(B^{\sigma-1}_{2,1})}
\Big(1+\|(\rho_{1},\rho_{2},\textbf{u}_{1},\textbf{u}_{2})\|_{B^{\sigma}_{2,1}}+\|(\theta_{1},\theta_{2})\|_{B^{\sigma+1}_{2,1}}\Big)dt.
\label{R-E50}
\end{eqnarray}
Gronwall's inequality gives
$(\widetilde{\rho},\widetilde{\textbf{u}},\widetilde{\textbf{E}})\equiv
\mathbf{0}$ immediately. Substituting it into (\ref{R-E45}),
$\widetilde{\theta}=0$ is also followed.

Finally, by Remark \ref{rem3.1}, we can arrive at Theorem
\ref{thm1.1} satisfying the inequality (\ref{R-E1001}). Hence the
proof of Theorem \ref{thm1.1} is complete.

\section{Appendix}\setcounter{equation}{0}
In the last section, we give the crucial Moser-type inequality in
the non-homogeneous Besov spaces and Chemin-Lerner's spaces. For the
homogeneous version, which has been remarked by Zhou in \cite{Z}.

\begin{prop}\label{prop4.1}
Let $s>0$ and $1\leq p,r,p_{1},p_{2},p_{3},p_{4}\leq\infty$. Assume
that $f\in L^{p_{1}}\cap B^{s}_{p_{4},r}$ and $g\in L^{p_{3}}\cap
B^{s}_{p_{2},r}$ with
$$\frac{1}{p}=\frac{1}{p_{1}}+\frac{1}{p_{2}}=\frac{1}{p_{3}}+\frac{1}{p_{4}}.$$
Then it holds that
\begin{eqnarray}
\|fg\|_{B^{s}_{p,r}}\leq
C(\|f\|_{L^{p_{1}}}\|g\|_{B^{s}_{p_{2},r}}+\|g\|_{L^{p_{3}}}\|f\|_{B^{s}_{p_{4},r}}).\label{R-E51}
\end{eqnarray}
In particular, whenever $s\geq N/p$, there holds
\begin{eqnarray}
\|fg\|_{B^{s}_{p,r}}\leq
C\|f\|_{B^{s}_{p,r}}\|g\|_{B^{s}_{p,r}}.\label{R-E52}
\end{eqnarray}

\end{prop}
\begin{proof}
From Bony's decomposition, we have
$$fg=T_{f}g+T_{g}f+R(f,g).$$
It follows from Proposition \ref{prop2.1} that
\begin{eqnarray}
2^{qs}\|\Delta_{q}T_{f}g\|_{L^p}&\leq&
2^{qs}\sum_{|q-q'|\leq4}\|\Delta_{q}(S_{q'}f\Delta_{q'}g)\|_{L^p}\nonumber\\&\leq&
\sum_{|q-q'|\leq4}2^{(q-q')s}\|S_{q'}f\|_{L^{p_{1}}}2^{q's}\|\Delta_{q'}g\|_{L^{p_{2}}}\Big(\frac{1}{p}=\frac{1}{p_{1}}+\frac{1}{p_{2}}\Big)
\nonumber\\&\leq&
C\sum_{|q-q'|\leq4}2^{(q-q')s}\|f\|_{L^{p_{1}}}2^{q's}\|\Delta_{q'}g\|_{L^{p_{2}}}\nonumber\\&\leq&
Cc_{q1}\|f\|_{L^{p_{1}}}\|g\|_{B^{s}_{p_{2},r}},\label{R-E53}
\end{eqnarray}
where
$c_{q1}:=\sum_{|q-q'|\leq4}\frac{2^{q's}\|\Delta_{q'}g\|_{L^{p_{2}}}}{9\|g\|_{B^{s}_{p_{2},r}}}$
satisfies $\|c_{q1}\|_{\ell^r}\leq1$. Similarly,
\begin{eqnarray}
2^{qs}\|\Delta_{q}T_{g}f\|_{L^p}&\leq&
2^{qs}\sum_{|q-q'|\leq4}\|\Delta_{q}(S_{q'}g\Delta_{q'}f)\|_{L^p}\nonumber\\&\leq&
Cc_{q2}\|g\|_{L^{p_{3}}}\|f\|_{B^{s}_{p_{4},r}},\label{R-E54}
\end{eqnarray}
where
$c_{q2}:=\sum_{|q-q'|\leq4}\frac{2^{q's}\|\Delta_{q'}f\|_{L^{p_{2}}}}{9\|f\|_{B^{s}_{p_{2},r}}}$
satisfies $\|c_{q2}\|_{\ell^r}\leq1$.

On the other hand, from Proposition \ref{prop2.2}, we arrive at
\begin{eqnarray}
&&2^{qs}\|\Delta_{q}R(f,g)\|_{L^p}\nonumber\\&\leq&Cc_{q3}\|R(f,g)\|_{B^{s}_{p,q}}
\nonumber\\&\leq&Cc_{q3}\|f\|_{B^{0}_{p_{1},\infty}}\|g\|_{B^{s}_{p_{2},q}}(s>0)
\nonumber\\&\leq&Cc_{q3}\|f\|_{L^{p_{1}}}\|g\|_{B^{s}_{p_{2},q}},\label{R-E55}
\end{eqnarray}
where
$c_{q3}:=\frac{2^{qs}\|\Delta_{q}R(f,g)\|_{L^p}}{\|R(f,g)\|_{B^{s}_{p,q}}}$
satisfies $\|c_{q3}\|_{\ell^r}\leq1$. In the last step, we used the
embedding property $B^{0}_{p_{1},1}\hookrightarrow
L^{p_{1}}\hookrightarrow B^{0}_{p_{1},\infty}$.

Hence, (\ref{R-E51}) follows from (\ref{R-E53})-(\ref{R-E55}).
Moreover, the embedding properties in Lemma \ref{lem2.2} give
(\ref{R-E52}) directly.
\end{proof}

It is not difficult to generalize Proposition \ref{prop4.1} to the
framework of Chemin-Lerner's spaces
$\widetilde{L}^{\rho}_{T}(B^{s}_{p,r})$. The indices $s,p,r$ behave
just as the stationary case whereas the time exponent $\rho$ behaves
according to H\"{o}lder's inequality, which is given by a
proposition for clarity.
\begin{prop}\label{prop4.2}
The following estimate holds:
$$
\|fg\|_{\widetilde{L}^{\rho}_{T}(B^{s}_{p,r})}\leq
C(\|f\|_{L^{\rho_{1}}_{T}(L^{p_{1}})}\|g\|_{\widetilde{L}^{\rho_{2}}_{T}(B^{s}_{p_{2},r})}
+\|g\|_{L^{\rho_{3}}_{T}(L^{p_{3}})}\|f\|_{\widetilde{L}^{\rho_{4}}_{T}(B^{s}_{p_{4},r})})
$$
whenever $s>0, \ 1\leq p,r\leq\infty,\ 1\leq p_{1},p_{2},p_{3},p_{4}
\leq\infty,\
1\leq\rho,\rho_{1},\rho_{2},\rho_{3},\rho_{4}\leq\infty$ with
$$\frac{1}{p}=\frac{1}{p_{1}}+\frac{1}{p_{2}}=\frac{1}{p_{3}}+\frac{1}{p_{4}},$$
and
$$\frac{1}{\rho}=\frac{1}{\rho_{1}}+\frac{1}{\rho_{2}}=\frac{1}{\rho_{3}}+\frac{1}{\rho_{4}}.$$
As a direct corollary, one has
$$\|fg\|_{\widetilde{L}^{\rho}_{T}(B^{s}_{p,r})}
\leq
C\|f\|_{\widetilde{L}^{\rho_{1}}_{T}(B^{s}_{p,r})}\|g\|_{\widetilde{L}^{\rho_{2}}_{T}(B^{s}_{p,r})}$$
whenever $s\geq N/p.$
\end{prop}

\section*{Acknowledgement}
The research of Jiang Xu is partially supported by the NSFC
(11001127), China Postdoctoral Science Foundation (20110490134) and
NUAA Research Funding (NS2010204).


\begin{thebibliography}{99}
\bibitem{A}
G. Al\`{\i}, Global existence of smooth solutions of the
$N$-dimensional Euler-Possion model, \textit{SIAM J. Math. Anal.},
{\bf{35}} (2003) 389-422.

\bibitem{ABN}
G. Al\`{\i}, D. Bini and R. Nionero, Global existence and relaxation
limit for smooth solutions to the Euler-Possion model for
semiconductors, \textit{SIAM J. Math. Anal.}, {\bf{32}} (2000)
572-587.

\bibitem{ACJP}
G.Al\`{\i}, L. Chen,  A. J\"{u}ngel, and Y. J. Peng, The
zero-electron-mass limit in the hydrodynamic model for plasmas,
\textit{Nonlinear Anal. TMA}, {\bf{72}} (2010) 4415-4427.


\bibitem{AVJM} P. Amster, M. P. Beccar Varela, A. J\"{u}ngel and
M. C. Mariani, Subsonic solutions to a one-dimensional
non-isentropic model for semiconductors, \textit{Journal of
Mathematical Analysis and Applications}, {\bf{258}} (2001)  52-62.

\bibitem {BCD}
H. Bahouri, J.~Y. Chemin and R. Danchin.\textit{ Fourier Analysis
and Nonlinear Partial Differential Equations}, Berlin, Heidelberg:
Springer-Verlag, 2011.

\bibitem {C} J.-Y. Chemin, Th\'{e}or\`{e}mes d'unicit\'{e} pour le syst\`{e}me de Navier-Stokes
tridimensionnel, \textit{Journal d'Analyse Math\'{e}matique},
{\bf{77}} (1999) 25-50.


\bibitem{CJZ} G. Q. Chen, J. W. Jerome and B. Zhang, Existence
and the singular relaxation limit for the inviscid hydrodynamic
energy model. \textit{Modeling and Computation for Applicition in
Mathematics, science, and Engineering }(Evanston, IL, 1996);
189--215, Numer. Math. Sci. Comput., Oxford Univ.Press: New York,
1998.

\bibitem {D}
R. Danchin, \textit{Fourier Analysis Methods for PDE's}, (Lecture
Notes), 2005.

\bibitem {D1}
R. Danchin, Global existence in critical spaces for compressible
Navier-Stokes equations, \textit{Inventiones Mathematicae}
{\bf{141}} (2000) 579-614.

\bibitem {D2}
R. Danchin, Local theory in critical spaces for flows of
compressible viscous and heat-conductive gases, \textit{Comm. P. D.
E.} {\bf{26}} (2001) 1183-1233.


\bibitem {DM} P. Degond and P. A. Markowich, A steady-state
potential flow model for semiconductors,  \textit{Ann. Mat. Pura
Appl.} {\bf{IV}} (1993) 87-98.


\bibitem {FXZ} D. Y. Fang, J. Xu and T. Zhang, Global
exponential stability of classical solutions to the hydrodynamic
model for semiconductors, \textit{Math. Models Methods Appl. Sci.}
{\bf{17}} (2007), 1507-1530.

\bibitem {Ga} I. Gamba, Stationary transonic solutions of a one-dimensional
hydrodynamic model for semiconductor, \textit{Comm. Partial Diff.
Equns} {\bf{17}} (1992), 553-577.

\bibitem{G} Y. Guo, Smooth irritational flows in the large
to the Euler-Poisson system in $\mathbb{R}^{3+1}$, \textit{Commun.
Math. Phys.} {\bf{195}} (1998) 249-265.

\bibitem{GS} Y. Guo and W. Strauss, Stability of semiconductor
states with insulating and contact boundary conditions,
\textit{Arch. Rational Mech. Anal.}, {\bf{179}} (2005), 1-30.

\bibitem{GN}
I. Gasser and R. Natalini, The energy transport and the drift
diffusion equations as relaxation limits of the hydrodynamic model
for semiconductors, \textit{Quart. Appl. Math.}, {\bf{57}} (1999),
269-282.

\bibitem{HJZ}
L. Hsiao, S. Jiang and P. Zhang, Global existence and exponential
stability of smooth solutions to a full hydrodynamic model to
semiconductors, \textit{Monatshefte f$\ddot{u}$r Mathematik},
{\bf{136}} (2002) 269-285.

\bibitem {HMW} L. Hsiao, P. A. Markowich and S. Wang, The asymptotic behavior of globally smooth solutions of the
multidimensional isentropic hydrodynamic model for semiconductors,
\textit{J. Differential Equations}, {\bf{192}} (2003) 111-133.

\bibitem{HW}
L. Hsiao and S. Wang, Asymptotic behavior of global smooth solutions
to the full 1D hydrodynamic model for semiconductors, \textit{Math.
Model Methods Appl. Sci.}, {\bf{12}} (2002) 777-796.

\bibitem{I}
D. Iftimie, The resolution of the Navier-Stokes equations in
anisotropic spaces, \textit{Revista Matem\'{a}tica Iberoamericana}
{\bf{15}} (1999), 1-36.


\bibitem{K} T. Kato, The Cauchy problem for quasi-linear
symmetric hyperbolic systems, \textit{Arch. Rational Mech. Anal.},
{\bf{58}} (1975) 181-205.

\bibitem{L}
Y. P. Li, Global existence and asymptotic behavior for a
multidimensional nonisentropic hydrodynamic semiconductor model with
the heat source, \textit{J. Differential Equations}, {\bf{225}}
(2006) 134-167.

\bibitem{LNX}
T. Luo, R. Natalini and Z. P. Xin, Large time behavior of the
solutions to a hydrodynamic model for semiconductors, \textit{SIAM
J. Appl. Math}, {\bf{59}} (1998) 810-830.

\bibitem{M} A. Majda, \textit{Compressible Fluid Flow and
Conservation laws in Several Space Variables} (Springer-Verlag:
Berlin/New York, 1984).


\bibitem{MN} P. Marcati and R. Natalini, Weak solutions to a
hydrodynamic model for semiconductors and relaxation to the
drift-diffusion equations, \textit{Arch. Ration. Mech. Anal.},
{\bf{129}} (1995) 129-145.

\bibitem{MRS} P. A. Markowich, C. Ringhofer and C. Schmeiser,
\textit{Semiconductor Equations}, Vienna, Springer-Verlag, 1990.

\bibitem{W}S. Wang, Quasineutral limit of Euler-Poisson system with and without
  viscosity.
\textit{Comm. Partial Differential Equations} {\bf{29}}, (2004)
419-456.

\bibitem{WC}
D. H. Wang and G. Q. Chen, Formation of singularities in
compressible Euler-Poisson fluids with heat diffusion and damping
relaxation, \textit{J. Differential Equations}, {\bf{144}} (1998)
44-65.

\bibitem{X1}
J. Xu, Energy-transport limit of the hydrodynamic model for
semiconductors, \textit{Math. Models Methods Appl. Sci.}, {\bf{20}}
(2010) 937-954.

\bibitem{X2}
J. Xu, Relaxation-time limit in the isothermal hydrodynamic model
for semiconductors, \textit{SIAM J. Math. Anal.}, {\bf{40}} (2009)
1979-1991.

\bibitem{X3}J. Xu, Well-posedness and stability of classical solutions to the
multidimensional full hydrodynamic model for semiconductors,
\textit{Comm. Pure Appl. Anal.}, {\bf{8}} (2009) 1073-1092.

\bibitem{XY1}J. Xu and W.-A. Yong, Relaxation-time limits of
non-isentropic hydrodynamic models for semiconductors, \textit{J.
Differential Equations}, {\bf{247}} (2009) 1777-1795.

\bibitem{XY2}
J. Xu and W.-A. Yong, Zero-relaxation limit of non-isentropic
hydrodynamic models for semiconductors, \textit{Discrete Contin.
Dyn. Syst.}, \textbf{25} (2009) 1319-1332.

\bibitem{XZ}
J. Xu and T. Zhang, Zero-electron-mass limit of Euler-Poisson
equations, submitted (2010).

\bibitem{Y} W.-A. Yong, Diffusive relaxation limit of multidimensional
isentropic hydrodynamical models for semiconductors, \textit{SIAM J.
Appl. Math.}, {\bf{64}} (2004) 1737-1748.

\bibitem{Z} Y. Zhou, Local well-posedness for the incompressible
Euler equations in the critical Besov spaces, \textit{Ann. Inst.
Fourier}, {\bf{54}} (2004) 773-786.

\end{thebibliography}
\end{document}